\documentclass[12pt]{article}

\usepackage{amsmath,amsfonts,amsthm,amssymb,graphicx,color}
\usepackage[mathscr]{euscript}
\usepackage{tikz}
\usepackage{authblk}

\renewcommand{\b}[1]{\boldsymbol{#1}}

\newcommand{\beps}{\b{\varepsilon}}
\newcommand{\bfx}{\b{x}}
\newcommand{\bn}{{\b{n}}}
\newcommand{\bnu}{{\b{\nu}}}
\newcommand{\bnabla}{\b{\nabla}}

\newcommand{\bsigma}{\b{\sigma}}

\newcommand{\btau}{\b{\tau}}

\newcommand{\bW}{\b{W}}
\newcommand{\bw}{\b{w}}

\newcommand{\cB}{\mathcal{B}}
\newcommand{\cE}{\mathcal{E}}

\newcommand{\cEnBN}{\mathcal{E}^{\mathrm{N}}_{\boldsymbol{n}}}

\newcommand{\cEnBE}{\mathcal{E}^{\mathrm{E}}_{\boldsymbol{n}}}
\newcommand{\cF}{\mathcal{F}}
\newcommand{\CF}{C_\mathrm{FP}}
\newcommand{\cG}{\mathcal{G}}
\newcommand{\cN}{\mathcal{N}}
\newcommand{\cT}{\mathcal{T}}
\newcommand{\CT}{C_\mathrm{T}}

\newcommand{\CTN}{C_\mathrm{T}^\mathrm{N}}
\newcommand{\ddiv}{\operatorname{div}}
\newcommand{\ds}{\dx[\b{s}]}
\newcommand{\dx}[1][\bfx]{\,\mathrm{d}#1}
\newcommand{\GammaD}{{\Gamma_{\mathrm{D}}}}
\newcommand{\GammaN}{{\Gamma_{\mathrm{N}}}}
\newcommand{\GammaNn}{{\Gamma_\bn^{\mathrm{N}}}}
\newcommand{\GammaNnp}{{\Gamma_\bn^{\mathrm{N}+}}}
\newcommand{\GammaNnz}{{\Gamma_\bn^{\mathrm{N}0}}}

\newcommand{\gN}{g_{\mathrm{N}}}
\newcommand{\Hdiv}[1][\Omega]{\b{H}(\ddiv,#1)}
\newcommand{\Ieff}{I_{\text{eff}}}

\newcommand{\norm}[1]{\left\|#1\right\|}
\newcommand{\oCT}{\overline{C}_\mathrm{T}}
\newcommand{\oGammaD}{\overline\Gamma_{\mathrm{D}}}
\newcommand{\oGammaN}{\overline\Gamma_{\mathrm{N}}}
\newcommand{\osc}{\operatorname{osc}}
\newcommand{\PiN}{\Pi_\mathrm{N}}
\newcommand{\R}{\mathbb{R}}
\newcommand{\RN}{R_\mathrm{N}}

\newcommand{\teta}{\tilde\eta}
\newcommand{\tbtau}{\tilde\btau}
\newcommand{\tGammaNK}{\widetilde\Gamma_{\mathrm{N}}^K}

\newcommand{\trinorm}[1]{|\!|\!|#1|\!|\!|}

\newtheorem{theorem}{Theorem}
\newtheorem{corollary}[theorem]{Corollary}
\newtheorem{lemma}[theorem]{Lemma}

\begin{document}

\title{A Simple Approach to Reliable and Robust A Posteriori Error Estimation
for Singularly Perturbed Problems} 

\author[a]{Mark Ainsworth}
\author[b]{Tom\'a\v s Vejchodsk\'y\thanks{T. Vejchodsk\'y acknowledges the support of the Czech Science Foundation, project no. 18-09628S, and the institutional support RVO 67985840.}}
\affil[a]{
\small
Division of Applied Mathematics, Brown University, 182 George St, Providence, RI~02912, USA and
Computer  Science  and  Mathematics  Division, Oak Ridge National Laboratory, Oak Ridge, TN~37831, USA,
\texttt{mark\_ainsworth@brown.edu}
}
\affil[b]{
  \small
  Institute of Mathematics,
  Czech Academy of Sciences,
  {\v Z}itn{\'a} 25, CZ-115 67 Prague 1,
  Czech Republic,
  \texttt{vejchod@math.cas.cz}
}

\maketitle

\begin{abstract} A simple flux reconstruction for finite element
solutions of reaction-diffusion problems is shown to yield fully computable
upper bounds on the energy norm of error in an approximation of singularly
perturbed reaction-diffusion problem. The flux reconstruction is based on
simple, independent post-processing operations over patches of elements in
conjunction with standard Raviart--Thomas vector fields and gives upper bounds
even in cases where Galerkin orthogonality might be violated. If Galerkin
orthogonality holds, we prove that the corresponding local error indicators are
locally efficient and robust with respect to any mesh size and any size of the
reaction coefficient, including the singularly perturbed limit. 
\end{abstract}

\noindent\textbf{Keywords:}
finite element analysis, robust a posteriori error estimate,
singularly perturbed problems, flux reconstruction


\noindent\textbf{MSC:} 65N15, 65N30, 65J15

\section{Introduction}
This paper is concerned with developing fully computable bounds for the error
in the finite element approximation of the following linear reaction-diffusion
problem
\begin{equation}
\label{eq:modpro}
  -\Delta u + \kappa^2 u = f \quad\text{in }\Omega; \qquad
  u = 0 \quad\text{on }\GammaD; \qquad
  \partial u/\partial \bnu = \gN \quad\text{on }\GammaN,
\end{equation}
where the domain $\Omega\subset\R^d$, $d\geq 2$, is a polytope and $\bnu$
denotes the unit outward normal vector on the boundary $\partial\Omega$. Here,
the portions $\GammaD$ and $\GammaN$ of the boundary $\partial\Omega$ are open,
disjoint and satisfy $\oGammaD\cup\oGammaN=\partial\Omega$. For simplicity, we
assume that the data $f\in L^2(\Omega)$ and $\gN\in L^2(\GammaN)$ and that the
reaction coefficient $\kappa\geq 0$ is piecewise constant.  The problem has a
unique solution provided that either $\GammaD$ has a positive measure or
$\kappa$ is not identically zero. 

Given a conforming approximation $u_h$ of the true solution $u$ of problem
\eqref{eq:modpro}, we present a novel a posteriori error estimator for the
energy norm of the error $\trinorm{u-u_h}$. The estimator is rather easy to
evaluate (via a fast element by element algorithm) and provides a guaranteed
upper bound on the true error measured in the energy norm. In the case where
$u_h$ is the Galerkin finite element approximation of $u$, we prove that the
estimator is locally efficient and provides an upper bound which does not
degenerate in the singularly perturbed limit (i.e. when $\kappa\to\infty$).



This error estimator is evaluated using a reconstructed flux.
In contrast to our previous work \cite{AinBab:1999,robustaee:2010,AinVej:2014},
the reconstructed flux is obtained by solving small local problems on patches of 
elements by Raviart--Thomas finite elements. This approach is technically simpler
and yields a more accurate error estimator.

The idea of the flux reconstruction by solving small problems on patches 
comes from \cite{BraSch:2008}. It can be seen as approximate minimization of 
the error bound using an overlapped domain decomposition method with subdomains
chosen as patches of elements. Local minimization problems on these patches 
have equilibration constraints and are solved by by mixed finite elements. Our 
result can be seen as a robust generalization of this idea to 
reaction-diffusion problems.

The general idea of flux reconstructions, however, dates back to the method of
hypercircle \cite{PraSyn:1947,Synge:1957} and later to
\cite{AubBur:1971,HasHla:1976,Kelly:1984,LadLeg:1983,Veubeke:1965}. In the last
two decades it was vastly developed, see
e.g. \cite{AinBab:1999,CaiZha:2010,HanSteVoh:2012,LucWoh:2004,ParSanDie2009,Rep:2008}
and references there in.  Interestingly, error estimates based on flux
reconstructions can be utilized to estimate various components of the error
such as the discretization, iteration, and algebraic errors
\cite{DolSebVoh:2015,ErnVoh:2010,JirStrVoh:2010,PapStrVoh:2018}. This enables
us to adaptively equilibrate all components of the error and develop algorithms
that do not perform excessive iterations of linear and nonlinear solvers in
cases when the iteration and algebraic errors are already on the level of the
discretization error.

The first robust, reliable, and locally efficient a posteriori error estimate
for problem \eqref{eq:modpro} was derived by Verf\"urth in
\cite{Verfurth:1998a,Verfurth:1998}.  
Locally efficient and robust guaranteed error bounds for the vertex-centred 
finite volume discretization of \eqref{eq:modpro} were proposed in 
\cite{CheFucPriVoh:2009}.  
Robust reliability estimate for singularly perturbed problem on anisotropic meshes is
proved in \cite{Kopteva:2017}. A similar result for a guaranteed and fully
computable error bound is provided in \cite{Kopteva:2018}. An interesting
alternative idea for guaranteed upper bounds for reaction-diffusion problems
was recently published in \cite{ParDie:2017}.
Preprint \cite{SmeVoh:2018} proofs robustness of a simple a posteriori error 
estimator, however 
their approach considerably differs from the one presented below due to 
equilibration of fluxes even if the reaction term dominates
and due to the presence of weights in the estimator.

The rest of the paper is organized as follows.  Section~\ref{se:modpro} briefly
introduces the finite element approximation of problem \eqref{eq:modpro} and
the corresponding notation.  Section~\ref{se:AEE} defines the a posteriori
error estimator and proves that it is the guaranteed upper bound on the error.
Section~\ref{se:locmin} introduces flux reconstruction based on local
minimization problems and Sections~\ref{se:sig1} and \ref{se:sig2} define two
auxiliary flux reconstructions that are used in Section~\ref{se:efficiency} to
prove the robust local efficiency of the proposed error indicators.
Section~\ref{se:noconstr} proposes alternative flux reconstruction that does
not require equilibration condition.  Section~\ref{se:numex} provides a couple
of numerical examples and Section~\ref{se:conclusions} draws the conclusions.

\section{Model Problem and Its Discretization} \label{se:modpro}
\subsection{Partitions}
Let $\cG = \{\cT_h\}$ be a family of partitionings of the domain $\Omega$ into
simplicial elements. The intersection of each distinct pair of elements in a
given partition $\cT_h \in \cG$ is assumed to consist of a single common vertex
or a single common facet of both elements.  The diameter and inradius of an
element $K$ are denoted by $h_K$ and $\rho_K$, respectively. The family $\cG$
is assumed to be regular in the sense that there exists a constant $C>0$ such
that
\begin{equation}
\label{eq:shapereg}
  \sup\limits_{\cT_h \in \cG} \max\limits_{K \in \cT_h} \frac{h_K}{\rho_K} \leq C.
\end{equation}
This assumption permits meshes in which the elements are locally refined such
as might arise from an adaptive refinement algorithm.  The patch consisting of
an element $K\in\cT_h$ and those elements in $\cT_h$ sharing at least one
common point with $K$ is defined by 
\begin{equation}
\label{eq:patchK}
  \widetilde K = \operatorname{int} \bigcup \left\{ 
  K'\in\cT_h : K' \cap K \neq \emptyset \right\}.
\end{equation}

The regularity condition \eqref{eq:shapereg} means that the number of elements
in any patch is uniformly bounded over the family $\cG$, as is the number of
patches containing a particular element.  Further, condition
\eqref{eq:shapereg} implies the following local quasi-uniformity and shape
regularity properties: there exist constants $c>0$ and $C>0$ such that for all
elements $K' \subset \widetilde K$, all $K\in\cT_h$, and all $\cT_h \in \cG$
estimates $ch_K \leq h_{K'} \leq C h_K$ and $c \rho_K \leq \rho_{K'} \leq C
\rho_K$ hold.

Here, and throughout, we adopt the convention whereby the symbol $C$ is used to
denote a generic constant throughout the paper, whose actual numerical value
can differ in different occurrences, but it is always independent of $\kappa$
and any mesh-size.

The notation $(\cdot,\cdot)_\omega$ and $\norm{\cdot}_\omega$ is used to denote
the $L^2(\omega)$ scalar product and norm over a subset $\omega\subset\Omega$,
and we omit the subscript in the case when $\omega=\Omega$.  The
$L^2(K)$-orthogonal projector onto the space of affine functions
$\mathbb{P}_1(K)$ over element $K\in\cT_h$ is denoted by
$\Pi_K:L^2(K)\rightarrow\mathbb{P}_1(K)$, whist $\Pi$ is used to denote the
concatenation of the  elementwise projections $\Pi_K$, i.e. $(\Pi f)|_K = \Pi_K
f$ for all $K\in\cT_h$.  Similarly, for a facet
$\gamma\subset\GammaN\cap\partial K$, $\Pi_\gamma : L^2(\gamma) \rightarrow
\mathbb{P}_1(\gamma)$ denotes the $L^2(\gamma)$-orthogonal projector, and
$\PiN$ denotes the concatenation of the facetwise projections
$\Pi_\gamma$: $(\PiN \gN)|_\gamma = \Pi_\gamma \gN$ for all facets
$\gamma \subset \GammaN$.

\subsection{Assumptions on the Reaction Coefficient $\kappa$}
\label{se:kappa} For simplicity, we shall assume that the reaction coefficient
is constant on every element over the entire set of partitions in $\cG$ and we
denote by $\kappa_K = \kappa|_K$ its constant value in $K \in \cT_h$.
Moreover, we shall assume that the reaction coefficient $\kappa$ varies slowly
between neighbouring elements in the sense that that there exists a constant
$C>0$ such that the following condition holds for all triangulations $\cT_h\in
\cG$ and all elements $K\in\cT_h$:
\begin{align}
  \label{eq:kappacond1}
  &\text{if}\quad h_K \kappa_K > 1 \quad\text{then}\quad
  \kappa_K \leq C \kappa_{K'}
  \quad \text{for all }K'\subset\widetilde K.
\end{align}
We state, without proof, some elementary consequences of the above assumption:
\begin{lemma}\label{le:kappa}
Suppose that condition~\eqref{eq:kappacond1} holds. Then
\begin{enumerate}
\item if $\kappa_K=0$, then $\kappa_{K'}<1/h_{K'}$ on the patch
      $\widetilde{K}$;
\item if $h_K \kappa_K > 1$, then $\kappa_{K'}>0$ on the patch $\widetilde{K}$; 
\item there exists a constant $C>0$ such that for all $\cT_h\in\cG$, all $K\in
\cT_h$, if $h_K \kappa_K > 1$, then
\begin{equation*}
C^{-1}\kappa_{K'}\leq\kappa_K\leq C\kappa_{K'} 
\end{equation*}
for all $K'\in\widetilde{K}$; 
\item there exists a constant $C>0$ such that for all $\cT_h\in\cG$, all $K\in
\cT_h$, and all elements $K' \subset \widetilde K$, 
\begin{equation*}
 C^{-1} \min\{h_{K'},\kappa_{K'}^{-1} \} \leq
        \min\{h_{K},\kappa_{K}^{-1} \} \leq
      C \min\{h_{K'},\kappa_{K'}^{-1} \}.
\end{equation*}
\end{enumerate}
\end{lemma}
The quantity $\min\{h_{K},\kappa_{K}^{-1} \}$ appears extensively throughout
the paper and we shall adopt the convention whereby 
\begin{equation}
\label{eq:kappa0}
   \min\{h_{K},\kappa_{K}^{-1} \} = h_K
   \quad\text{if }\kappa_K = 0.
\end{equation}

\subsection{Finite Element Discretization}
The weak formulation of problem \eqref{eq:modpro} reads:
find $u\in V = \{ v \in H^1(\Omega) : v = 0 \text{ on }
\GammaD \}$ such that
\begin{equation}
\label{eq:weakf}
  \cB(u, v) = \cF(v) \quad \forall v \in V,
\end{equation}
where $\cB:V\times V\to\R$ and $\cF:V\to\R$ are defined by 
\begin{equation*}
  \cB(u,v) = \int_\Omega (\bnabla u \cdot \bnabla v + \kappa^2 u v ) \dx; \quad
  \cF(v) = \int_\Omega f v \dx + \int_\GammaN \gN v \ds.
\end{equation*}
It will be useful to introduce local counterparts of these forms 
\begin{equation*}
  \cB_K(u,v) = \int_K (\bnabla u \cdot \bnabla v + \kappa_K^2 u v ) \dx; \quad
  \cF_K(v) = \int_K f v \dx + \int_{\GammaN\cap\partial K} \gN v \ds.
\end{equation*}
The associated global and local energy norms $\trinorm{\cdot}$ and
$\trinorm{\cdot}_K$ are defined by $\trinorm{v}^2 = \cB(v,v)$ and
$\trinorm{v}_K^2 = \cB_K(v,v)$, respectively. 

Let $V_h = \{ v_h \in V : v_h|_K \in \mathbb{P}_1(K) \ \forall K \in \cT_h \}$,
where $\mathbb{P}_1(K)$ is the space of affine functions on $K$, then the
finite element approximation $u_h \in V_h$ of \eqref{eq:modpro} is defined by 
\begin{equation}
  \label{eq:FEM}
  \cB(u_h, v_h ) = \cF(v_h) \quad \forall v_h \in V_h.
\end{equation}

\section{A Posteriori Error Estimator}\label{se:AEE}
Every partition $\cT_h\in\cG$ can be split into disjoint subsets $\cT_h^+ = \{
K \in \cT_h : \kappa_K > 0 \}$ and $\cT_h^0 = \{ K \in \cT_h : \kappa_K = 0
\}$. Let $\btau\in\Hdiv$ be any vector field satisfying the conditions 
\begin{align}
  \label{eq:equilib1}
  -\ddiv\btau &= \Pi f - \kappa^2 u_h \text{ in all elements } K \in \cT_h^0,
\\ \label{eq:equilib2} 
  \btau\cdot\bnu &= \PiN\gN \mbox{ on all facets } \gamma\subset\GammaN \cap \partial K,\ K\in\cT_h^0.
\end{align}
Let $\beps = \btau - \bnabla u_h$ in $\Omega$, 
$r = \Pi f - \kappa^2 u_h + \ddiv \btau$ in $\Omega$, and
$\RN = \PiN \gN - \btau \cdot \bnu$ on $\GammaN$,
then the local error indicator over an element $K\in\cT_h$ is defined to be 
\begin{equation}
\label{eq:etaK}
  \eta_K(\btau) = \left( \norm{\beps}_K^2
      + \kappa_K^{-2} \norm{r}_K^2 \right)^{1/2} 
      + \sum_{\gamma \subset \GammaN\cap\partial K} \CT^{K,\gamma} \norm{\RN}_\gamma.
\end{equation}
Observe that $r$ and $\RN$ vanish if $K\in\cT_h^0$ and that the second and the
third term is taken to be zero on such elements. The error estimator is then 
defined by 
\begin{equation}
\label{eq:eta}
  \eta^2(\btau) = \sum_{K\in\cT_h} [\eta_K(\btau) + \osc_K(f,\gN)]^2
\end{equation}
where the oscillation term is given by 
\begin{equation*}
\osc_K(f,\gN) = \min \left\{ \frac{h_K}{\pi}, \frac{1}{\kappa_K} \right\}
   \norm{f - \Pi_K f}_K
   + \sum_{\gamma \subset \GammaN\cap\partial K} \min\{\CT^{K,\gamma},\oCT^{K,\gamma}\} \norm{\gN - \Pi_\gamma \gN}_{\gamma}
\end{equation*}
and the constants 
\begin{align*}
  \left(\CT^{K,\gamma}\right)^2 &=  \frac{|\gamma|}{d |K|} \frac{1}{\kappa_K}
    \sqrt{(2h_K)^2 + (d/\kappa_K)^2},
\\
  \left(\oCT^{K,\gamma}\right)^2 &= \frac{|\gamma|}{d |K|} \min\{h_K/\pi,\kappa_K^{-1}\}
    \left(2h_K + d \min\{h_K/\pi,\kappa_K^{-1}\} \right)
\end{align*}
arose in the corrigendum of \cite[Lemma~1]{AinVej:2014}.

The following result, based on \cite[Lemma~2]{AinVej:2014}, shows that the estimator 
provides an upper bound on the error: 
\begin{theorem}
\label{th:main}
Let $u_h \in V$ be arbitrary. 
If $\btau \in\Hdiv$ satisfies equilibration conditions \eqref{eq:equilib1}--\eqref{eq:equilib2}
then
\begin{equation}
\label{eq:upperbound}
\trinorm{u - u_h} \leq 
  \eta(\btau). 
\end{equation}
\end{theorem}
\begin{proof}
The weak formulation \eqref{eq:weakf} and the divergence theorem yield identity
\begin{multline}
\label{eq:errid}
  \cB(u - u_h,v) 
= \sum\limits_{K\in\cT_h} \left[
  (\beps, \bnabla v)_K
+ ( r, v )_K
+ \sum\limits_{\gamma\subset\GammaN\cap\partial K}
  ( \RN, v )_\gamma
\right.
\\
\left.
+ (f - \Pi_K f, v)_K
+ \sum\limits_{\gamma\subset\GammaN\cap\partial K} (\gN - \Pi_\gamma \gN, v)_\gamma
\right]
\end{multline}
for all $v \in V$.
 The last two terms are estimated in the same way as in the proof of \cite[Lemma~2]{AinVej:2014}:
\begin{equation}
  \label{eq:oscKest}
  (f - \Pi_K f, v)_K
+ \sum\limits_{\gamma\subset\GammaN\cap\partial K} (\gN - \Pi_\gamma \gN, v)_\gamma
 \leq
 \osc_K(f,\gN) \trinorm{v}_K.
\end{equation}
For elements $K\in\cT_h^+$ we bound
\begin{gather}
  \label{eq:epsest}
  ( \beps, \bnabla v)_K + ( r, v )_K
  \leq \left( \norm{\beps}_K^2 + \kappa_K^{-2}\norm{r}_K^2 \right)^{1/2} \trinorm{v}_K,
\\
  \label{eq:RNest}
  ( \RN, v)_\gamma \leq \CT^{K,\gamma} \norm{\RN}_\gamma \trinorm{v}_K,
\end{gather}
where the trace inequality \cite[Lemma~1]{AinVej:2014} is employed.

Due to equilibration conditions \eqref{eq:equilib1}--\eqref{eq:equilib2},
we arrive at
\begin{multline*}
  \cB(u - u_h,v)
  \leq \sum_{K\in\cT_h^+} \left[
    \left( \norm{\beps}_K^2 + \kappa_K^{-2}\norm{r}_K^2 \right)^{1/2}
    + \sum\limits_{\gamma\subset\GammaN\cap\partial K} \CT^{K,\gamma} \norm{\RN}_\gamma
  \right] \trinorm{v}_K
\\
  + \sum_{K\in\cT_h^0} 
   \norm{\beps}_K \trinorm{v}_K
  + \sum\limits_{K\in\cT_h} \osc_K(f,\gN) \trinorm{v}_K
=
  \sum_{K\in\cT_h} [\eta_K(\btau) + \osc_K(f,\gN)] \trinorm{v}_K
\end{multline*}
Cauchy--Schwarz inequality, notation \eqref{eq:eta}, and choice $v=u-u_h$ finish the proof.
\end{proof}
It will not have escaped the reader's notice that nothing in the above argument
relies on $u_h$ being a finite element approximation. Consequently, the upper
bound presented in Theorem~\ref{th:main} holds true for arbitrary conforming
approximation $u_h\in V$. However, the local efficiency and robustness results
proved in Theorem~\ref{th:loceff} will require $u_h$ to be a Galerkin finite
element approximation exactly satisfying the condition \eqref{eq:FEM}.

\section{Flux Reconstruction by Patchwise Minimization}\label{se:locmin}
Let $\cN_h$ denote the nodes in the partition $\cT_h$.  In particular, given a
node $\bn\in\cN_h$, the subset $\cT_\bn = \{ K \in \cT_h : \bn \in K \}$
consists of elements that touch the node, while
$\cEnBN=\{\gamma\subset\GammaN: \bn\in\gamma\}$ consists of facets on the
Neumann boundary $\GammaN$ which touch the node $\bn$. 

The flux reconstructions used in the current work are constructed over patch
$\omega_\bn = \operatorname{int} \bigcup \cT_\bn$. Specifically, let 
\begin{multline}
\label{eq:defWomegan}
  \bW(\omega_\bn) = \{ \btau \in \Hdiv[\omega_\bn] : \btau|_K \in \mathbf{RT}_1(K),\
    \btau\cdot\bnu_\bn = 0 \text{ on } \gamma \in \cEnBE
    \},
\end{multline}
where $\cEnBE=\{\gamma\subset\partial\omega_\bn : \bn \not\in \gamma \}$,
$\bnu_\bn$ denotes the unit outward facing normal vector on the boundary of the
patch $\omega_\bn$, and $\mathbf{RT}_1(K) = [\mathbb{P}_1(K)]^d \oplus \bfx
\mathbb{P}_1(K)$ is the standard Raviart--Thomas space.

Let $\btau_\bn \in \bW(\omega_\bn)$ denote the minimizer of the quadratic functional
\begin{multline}
  \label{eq:taunmin}
  E_\bn(\btau_\bn) = \norm{ \btau_\bn - \theta_\bn \bnabla u_h }_{\omega_\bn}^2 \\   
  + \norm{ \kappa^{-1} \left[ \Pi(\theta_\bn (\Pi f - \kappa^2 u_h))
  - \bnabla \theta_\bn \cdot \bnabla u_h + \ddiv \btau_\bn \right] }_{\omega_\bn^+}^2 \\
  + \norm{ \CTN \left[ \PiN(\theta_\bn \PiN \gN) - \btau_\bn \cdot\bnu \right] }_\GammaNnp^2, 
\end{multline}
over $\btau_\bn\in\bW(\omega_\bn)$ satisfying constraints
\begin{align}
  \label{eq:tauconstr1}
  \Pi(\theta_\bn (\Pi f - \kappa^2 u_h))
    - \bnabla \theta_\bn \cdot \bnabla u_h + \ddiv \btau_\bn
  &= 0 
  \quad\text{in }\omega_\bn^0,
\\  
  \label{eq:tauconstr2}
  \PiN(\theta_\bn \PiN \gN) - \btau_\bn \cdot\bnu &= 0
  \quad\text{on }\GammaNnz,
\end{align}
where $\theta_\bn$ is the usual piecewise affine and continuous hat function satisfying
$\theta_\bn(\bn') = \delta_{\bn\bn'}$ for all nodes $\bn' \in \cN_h$,
$\omega_\bn^+ = \operatorname{int} \bigcup (\cT_\bn \cap \cT_h^+)$,
$\omega_\bn^0 = \operatorname{int} \bigcup (\cT_\bn \cap \cT_h^0)$,
$\GammaNnp = \bigcup \{ \gamma \in \cEnBN : \kappa_{K_\gamma} > 0 \}$,
$K_\gamma$ is the element adjacent to the facet $\gamma$,
$\GammaNnz = \bigcup \{ \gamma \in \cEnBN : \kappa_{K_\gamma} = 0 \}$,
and
$\CTN$ stands for piecewise constant function over facets defined as
$\CTN|_\gamma = \CT^{K_\gamma,\gamma}$ for all facets $\gamma \subset \GammaN$
such that $\kappa_{K_\gamma} > 0$.

The minimizer of \eqref{eq:taunmin} 
satisfying constraints \eqref{eq:tauconstr1}--\eqref{eq:tauconstr2} could, 
equally well, be characterised as the unique solution of the following problem:
Find $\btau_\bn \in \bW(\omega_\bn)$ and Lagrange multipliers
$q_h \in \mathbb{P}_1^*(\omega_\bn^0)$ and $d_h \in \mathbb{P}_1^*(\GammaNnz)$
satisfying
\begin{multline}
  \label{eq:taun}
  (\btau_\bn, \bw_h)_{\omega_\bn} 
  + (\kappa^{-2} \ddiv \btau_\bn, \ddiv \bw_h)_{\omega_\bn^+}
  + ( (\CTN)^2 \btau_\bn\cdot\bnu,\bw_h\cdot\bnu)_\GammaNnp
\\
  + (q_h, \ddiv \bw_h)_{\omega_\bn^0} 
  + (d_h, \bw_h \cdot \bnu)_\GammaNnz
  = (\theta_\bn \bnabla u_h, \bw_h)_{\omega_\bn}
\\  
  - \left( \kappa^{-2}\left[ \theta_\bn (\Pi f - \kappa^2 u_h) - \bnabla \theta_\bn \cdot \bnabla u_h \right], \ddiv \bw_h \right)_{\omega_\bn^+}
  + ( (\CTN)^2 \theta_\bn\PiN\gN,\bw_h\cdot\bnu)_\GammaNnp
\end{multline}
for all $\bw_h \in \bW(\omega_\bn)$ and
\begin{align}
 \label{eq:taunconstr1}
 (\ddiv\btau_\bn, \varphi_h)_{\omega_\bn^0} 
 &= (\bnabla \theta_\bn \cdot \bnabla u_h - \theta_\bn (\Pi f - \kappa^2 u_h), \varphi_h)_{\omega_\bn^0}
 \quad\forall \varphi_h \in \mathbb{P}_1^*(\omega_\bn^0),
\\
 \label{eq:taunconstr2}
 (\btau_\bn \cdot \bnu,\psi_h)_\GammaNnz
 &= (\theta_\bn\PiN\gN, \psi_h)_\GammaNnz
 \quad\forall \psi_h \in \mathbb{P}_1^*(\GammaNnz),
\end{align}
where $\mathbb{P}_1^*(\omega_\bn^0)$ and $\mathbb{P}_1^*(\GammaNnz)$ 
are spaces of discontinuous and piecewise affine functions over $\omega_\bn^0$
and $\GammaNnz$, respectively.

The condition in definition \eqref{eq:defWomegan} imposed on the facets
$\cEnBN$ means that $\btau_\bn$ can be extended by zero onto $\Omega$ thereby
obtaining a vector field in $\Hdiv$, which we again denote by $\btau_\bn$.
With this convention in place, the reconstructed flux $\btau \in \Hdiv$ is
taken to be the sum 
\begin{equation}
  \label{eq:tau}
  \btau = \sum_{\bn \in \cN_h} \btau_\bn.
\end{equation}
The resulting globally defined vector field $\btau$ can be used in
\eqref{eq:upperbound} to obtain an upper bound on the energy norm of the error,
because it satisfies the equilibration conditions as stated in the following lemma.
\begin{lemma}
Reconstructed flux $\btau\in\Hdiv$ given by \eqref{eq:tau} satisfies
equilibration conditions \eqref{eq:equilib1}--\eqref{eq:equilib2}.
\end{lemma}
\begin{proof}
Let $K\in\cT_h^0$.
Equality \eqref{eq:taunconstr1}, definition \eqref{eq:tau}, and partition of unity $\theta_\bn$
yield
$$
  0 = \sum_{\bn \in \cN_K} (\theta_\bn (\Pi_K f - \kappa_K^2 u_h) 
    - \bnabla \theta_\bn \cdot \bnabla u_h + \ddiv \btau_\bn, \varphi_h)_K
  = 
  (\Pi_K f - \kappa_K^2 u_h + \ddiv \btau, \varphi_h)_K
$$
for all $\varphi_h \in \mathbb{P}_1(K)$. Since $\Pi_K f - \kappa_K^2 u_h|_K + \ddiv \btau|_K \in \mathbb{P}_1(K)$,
equilibration condition \eqref{eq:equilib1} follows.

Similarly, given a facet $\gamma \subset \GammaN\cap\partial K$, the equality \eqref{eq:taunconstr2} implies
$$
  0 = \sum_{\bn \in \cN_\gamma}
   (\theta_\bn\Pi_\gamma\gN - \btau_\bn \cdot \bnu, \psi_h)_\gamma
  = (\Pi_\gamma\gN - \btau \cdot \bnu, \psi_h)_\gamma
$$
for all $\psi_h \in \mathbb{P}_1(\gamma)$. Equilibration condition \eqref{eq:equilib2} then follows, 
because $\Pi_\gamma\gN - \btau \cdot \bnu|_\gamma \in \mathbb{P}_1(\gamma)$.
\end{proof}

The next two sections are concerned with showing that reconstructed flux $\btau$ 
defined in \eqref{eq:tau} yields locally
efficient and robust error indicators. 
The main idea used in the proof is based on comparing $\btau$ with two
judiciously chosen flux reconstructions $\bsigma^{(1)}_{\widetilde K}$ and
$\bsigma^{(2)}_{\widetilde K}$. Each of these reconstructions is defined in the
neighbourhood $\widetilde K$ of the element $K$, see \eqref{eq:patchK}.  While
the flux reconstruction $\bsigma^{(1)}_{\widetilde K}$ is based on equilibrated
interface fluxes $g_K$ introduced in \cite{AinBab:1999} and analysed in
\cite{AinVej:2014}, the second reconstruction $\bsigma^{(2)}_{\widetilde K}$ is
new.

\section{The First Auxiliary Flux Reconstruction}
\label{se:sig1}

In this section we introduce the auxiliary flux reconstruction $\bsigma^{(1)}_{\widetilde K}$ and prove its properties.
This flux reconstruction is based on equilibrated interface fluxes. 
If $u_h \in V_h$ is the Galerkin solution given by \eqref{eq:FEM} then results of \cite{AinOde:2000} guarantee the existence of interface fluxes $g_K$ satisfying for all elements $K\in\cT_h$ the following properties
\begin{alignat}{2}
\nonumber 
g_K|_\gamma &\in \mathbb{P}_1(\gamma)
&\quad&\text{for all facets } \gamma \subset \partial K,
\\ \label{eq:gKNeu}
g_K &= \Pi_\gamma \gN
&\quad&\text{for all facets } \gamma \subset \GammaN\cap\partial K,
\\
\label{eq:gKconsist}
  g_K + g_{K'} &= 0
  &\quad&\text{on facets } \gamma = \partial K \cap \partial K'
  \text{ for some } K' \in \cT_h,
\end{alignat}
and equilibration condition
\begin{equation}
\label{eq:gKequilib}
 \int_K f \theta_\bn \dx 
  - \cB_K(u_h,\theta_\bn)
  + \int_{\partial K} g_K \theta_\bn \ds = 0
  \quad\text{for all } 
  \bn \in \cN_K,
\end{equation}
where $\cN_K$ stands for the set of $d+1$ vertices of $K$.
These fluxes do not yield robust a posteriori error estimators for large values of the reaction coefficient $\kappa$, as it was shown in \cite{AinBab:1999}. However, we will use them only in elements where $h_K \kappa_K \leq 1$.

Below, we will utilize two estimates from \cite{AinVej:2014}.
First, quantity $R = g_K - \bnabla u_h \cdot \bnu_K$ defined on $\partial K$ for all elements $K \in \cT_h$ 
satisfies 
\begin{equation}
  \label{eq:Rest}
  \norm{ R }_{\partial K} 
  \leq C \left[ h_K^{-1/2} \trinorm{u-u_h}_{\widetilde K} + h_K^{1/2} \norm{f - \Pi f}_{\widetilde K} +
    \norm{\gN - \Pi_{\GammaN} \gN}_{\GammaN\cap\partial K} \right],
%
\end{equation}
see \cite[estimate (31)]{AinVej:2014}.
Second, residual $r_h = \Pi f - \kappa^2 u_h + \Delta u_h$ is bounded as
\begin{equation}
  \label{eq:rKest}
  \norm{ r_h }_K \leq C \left[ \min\{h_K,\kappa_K^{-1}\}^{-1} \trinorm{u - u_h}_K + \norm{f - \Pi_K f}_K \right],
\end{equation}
see \cite[estimate (29)]{AinVej:2014} and also \cite[Lemma~5]{AinBab:1999}.

The definition of the first auxiliary flux reconstruction proceeds as follows.
If an element $K\in\cT_h$ is such that $\kappa_K h_K \leq 1$, then we define
\begin{equation}
\label{eq:sigma1}
  \bsigma^{(1)}_{\widetilde K} = \sum_{\bn \in \cN_K} \bsigma_\bn^{(1)},
\end{equation}
where $\bsigma_\bn^{(1)}$ is given piecewise as
\begin{equation}
\label{eq:sigma1n}
  \bsigma^{(1)}_\bn|_K = 
    \bsigma^{(1)}_{\bn,K}  \text{ for all } K \in \cT_\bn. \\
\end{equation}
and vector fields $\bsigma^{(1)}_{\bn,K}$ are determined by the following lemma.

\begin{lemma}
\label{le:sigmanK}
Let $u_h \in V_h$ satisfies \eqref{eq:FEM}.
Let $K\in\cT_h$ be an element and let $\bn \in \cN_K$ be its vertex.
Then there exists $\bsigma^{(1)}_{\bn,K} \in \mathbf{RT}_1(K)$ such that
\begin{equation}
  \label{eq:sigma1nKnuK}
  \bsigma^{(1)}_{\bn,K} \cdot \bnu_K = \Pi_\gamma (\theta_\bn g_K)
  \quad\text{on facets } \gamma \subset \partial K
\end{equation}
and 
\begin{equation}
  \label{eq:divsigma1nK}
  -\ddiv\bsigma^{(1)}_{\bn,K} = 
      \Pi_K \theta_\bn (\Pi_K f - \kappa^2 u_h) - \bnabla \theta_\bn \cdot \bnabla u_h
  \quad\text{in }K.
\end{equation}
\end{lemma}
\begin{proof}
To establish the existence and uniqueness of $\bsigma^{(1)}_{\bn,K}$ we recall 
\cite{BreFor:1991} that $\mathbf{RT}_1(K)$ is unisolvent with respect to the
degrees of freedom defined by 
\begin{equation*}
   \bsigma\to\int_K \bsigma\cdot\b{v}, \quad\b{v}\in\mathbb{P}_0^d(K)
\end{equation*}
and 
\begin{equation}\label{mark2}
   \bsigma\to\int_\gamma \bn\cdot\bsigma w, \quad w\in\mathbb{P}_1(\gamma). 
\end{equation}
Observing that $\nabla:\mathbb{P}_1(K)/\R\to\mathbb{P}_0^d(K)$ is surjective, we
may rewrite the former set of degrees of freedom in the equivalent form 
\begin{equation*}
   \bsigma\to\int_K \bsigma\cdot\nabla{v}, \quad v\in\mathbb{P}_1(K)/\R, 
\end{equation*}
which, on integrating by parts and using the second set of degrees of freedom,
shows that $\mathbf{RT}_1(K)$ is unisolvent with respect to the degrees of 
freedom defined by~\eqref{mark2} augmented with the following 
\begin{equation*}
   \bsigma\to\int_K \ddiv\bsigma v, \quad v\in\mathbb{P}_1(K)/\R. 
\end{equation*}
Since the data in conditions
\eqref{eq:sigma1nKnuK} belong to $\mathbb{P}_1(\gamma)$
and since the equilibration condition \eqref{eq:gKequilib} implies
the following compatibility condition
\begin{multline*}
     \int_K \Pi_K \theta_\bn (\Pi_K f - \kappa^2 u_h) \dx 
   - \int_K \bnabla \theta_\bn \cdot \bnabla u_h \dx 
  + \sum_{\gamma \subset \partial K} \int_\gamma \Pi_\gamma (\theta_\bn g_K) \ds 
\\  
  =
     \int_K \theta_\bn (\Pi_K f - \kappa^2 u_h) \dx 
   - \int_K \bnabla \theta_\bn \cdot \bnabla u_h \dx 
  + \int_{\partial K} \theta_\bn g_K \ds 
\\  
  =
\int_K f \theta_\bn \dx 
  - \cB_K(u_h,\theta_\bn)
  + \int_{\partial K} g_K \theta_\bn \ds = 0,
\end{multline*}
we deduce that $\bsigma^{(1)}_{\bn,K}$ exists and is unique. 
\end{proof}

The following lemma shows that vector fields $\bsigma^{(1)}_\bn$ lie in $\bW(\omega_\bn)$.
\begin{lemma}
\label{le:sigma1inW}
Let $u_h \in V_h$ satisfies \eqref{eq:FEM}.
Let $\bn \in \cN_h$ be a vertex
and let $\bsigma^{(1)}_\bn$ be defined by \eqref{eq:sigma1n}. Then $\bsigma^{(1)}_\bn \in \bW(\omega_\bn)$ and $\bsigma^{(1)}_\bn \cdot \bnu_\bn = \Pi_\gamma (\theta_\bn \Pi_\gamma g_N)$ on facets $\gamma \in \cEnBN$.
\end{lemma}
\begin{proof}
The fact that $\bsigma^{(1)}_\bn \in \Hdiv[\omega_\bn]$ follows from the continuity of its normal components over element interfaces. Indeed, 
if $\gamma \subset \partial K \cap \partial K'$ for elements $K,K' \in \cT_\bn$ is an interior facet 
then
$$
  \bsigma^{(1)}_{\bn,K} \cdot \bnu_K + \bsigma^{(1)}_{\bn,K'} \cdot \bnu_{K'} 
  = \Pi_\gamma(\theta_\bn g_K) + \Pi_\gamma(\theta_\bn g_{K'})
  = \Pi_\gamma \theta_\bn (g_K + g_{K'}) = 0
$$
by \eqref{eq:gKconsist}.
Similarly, we verify the boundary conditions on $\partial\omega_\bn$ required in \eqref{eq:defWomegan}.
Clearly, $\bsigma^{(1)}_\bn \cdot \bnu_\bn = \Pi_\gamma(\theta_\bn g_K) = 0$ on facets $\gamma \in \cEnBE$, because $\theta_\bn = 0$ on $\gamma$,
and $\bsigma^{(1)}_\bn \cdot \bnu_\bn = \Pi_\gamma (\theta_\bn \Pi_\gamma g_N)$ on facets $\gamma \in \cEnBN$ by \eqref{eq:gKNeu}.
\end{proof}

An interesting consequence of Lemma~\ref{le:sigma1inW} and identity \eqref{eq:divsigma1nK} is that
\begin{equation}
  \label{eq:Ensig1}
  E_\bn\left(\bsigma^{(1)}_\bn\right) = \norm{\bsigma^{(1)}_\bn - \theta_\bn \bnabla u_h}_{\omega_\bn}^2.
\end{equation}
The vanishing normal components of $\bsigma^{(1)}_\bn$ on exterior facets $\gamma \in \cEnBE$ guarantee that $\bsigma^{(1)}_\bn$ can be extended by zero and as such belongs to $\Hdiv$. 
Consequently, $\bsigma^{(1)}_{\widetilde K}$ defined in \eqref{eq:sigma1} lies in $\Hdiv$ as well.
The following lemma shows that $\bsigma^{(1)}_{\widetilde K}$ satisfies equilibration 
conditions \eqref{eq:equilib1}--\eqref{eq:equilib2}.

\begin{lemma}
\label{le:sigmaprop}
Let $u_h \in V_h$ satisfies \eqref{eq:FEM}.
Let $K\in\cT_h$ be a fixed element and
let $\bsigma^{(1)}_{\widetilde K} \in \Hdiv$ be defined by \eqref{eq:sigma1}. 
Then
\begin{alignat}{2}
  \label{eq:divsigma}
  \Pi_K f - \kappa_K^2 u_h + \ddiv \bsigma^{(1)}_{\widetilde K} &= 0
  &\quad &\text{in } K,
\\
  \label{eq:sigmaNeu}
  \Pi_\gamma \gN - \bsigma^{(1)}_{\widetilde K} \cdot \bnu &= 0 &\quad &\text{on all facets }\gamma \subset \GammaN \cap \partial K.
\end{alignat}
\end{lemma}
\begin{proof}
To prove \eqref{eq:divsigma}, we use
definitions \eqref{eq:divsigma1nK} and \eqref{eq:sigma1n} to find that
\begin{equation}
  \label{eq:divsigman}
  \theta_\bn(\Pi_K f - \kappa_K^2 u_h) - \bnabla \theta_\bn \cdot \bnabla u_h + \ddiv \bsigma^{(1)}_\bn
  = \theta_\bn(\Pi_K f - \kappa_K^2 u_h) - \Pi_K \theta_\bn(\Pi_K f - \kappa_K^2 u_h)
\end{equation}
holds in $K$.
This identity together with the partition of unity $\sum_{\bn \in \cN_K} \theta_\bn = 1$, definition \eqref{eq:sigma1}, and properties of the projection $\Pi_K$ yields
\begin{multline*}
   \Pi_K f - \kappa_K^2 u_h + \ddiv \bsigma^{(1)}_{\widetilde K} 
 = \sum_{\bn \in \cN_K} \left[ \theta_\bn (\Pi_K f - \kappa_K^2 u_h) 
   - \bnabla \theta_\bn \cdot \bnabla u_h + \ddiv \bsigma^{(1)}_\bn \right]
\\   
 = \sum_{\bn \in \cN_K} \left[ \theta_\bn (\Pi_K f - \kappa_K^2 u_h) 
   - \Pi_K \theta_\bn (\Pi_K f - \kappa_K^2 u_h) \right]
\\   
 = \Pi_K f - \kappa_K^2 u_h
   - \Pi_K (\Pi_K f - \kappa_K^2 u_h)
 = 0
\end{multline*}
in $K$.

To prove \eqref{eq:sigmaNeu}, we consider
$K\in\cT_h$ to be an element adjacent to the Neumann boundary $\GammaN$ and
$\gamma \subset \GammaN \cap \partial K$ to be its facet. On this $\gamma$ we clearly have
$$
  \bsigma^{(1)}_{\widetilde K} \cdot \bnu = \sum_{\bn \in \cN_\gamma} \bsigma^{(1)}_{\bn,K} \cdot \bnu _K
  = \sum_{\bn \in \cN_\gamma} \Pi_\gamma (\theta_\bn \Pi_\gamma \gN)
  = \Pi_\gamma \gN
$$
by \eqref{eq:sigma1}, \eqref{eq:sigma1n}, \eqref{eq:sigma1nKnuK}, \eqref{eq:gKNeu},
and the fact that $\sum_{\bn \in \cN_\gamma} \theta_\bn = 1$.
Here, $\cN_\gamma$ stands for the set of $d$ vertices of the facet $\gamma$.
\end{proof}

Now, we formulate and prove the main result of this section.
For an element $K\in\cT_h$, we introduce neighbourhood $\widetilde{\widetilde K} = \bigcup \{ K' \in \cT_h : K' \cap \widetilde K \neq 0 \}$,
where $\widetilde K$ is given by \eqref{eq:patchK}. 
\begin{theorem}
\label{th:sigma}
Let $K\in\cT_h$ be an element where $\kappa_K h_K \leq 1$. 
Let $\bn \in \cN_K$ be its vertex and
let $\bsigma^{(1)}_\bn$ be defined by \eqref{eq:sigma1n}.
Then 
\begin{equation}
  \label{eq:etaKsigma_eff}
  \norm{ \bsigma^{(1)}_\bn - \theta_\bn \bnabla u_h }_{\omega_\bn}^2
\\  
   \leq C \left[ \trinorm{u-u_h}_{\widetilde{\widetilde K}}^2 
    + h_K^2 \norm{f - \Pi f}_{\widetilde{\widetilde K}}^2 
    + h_K \sum_{\gamma\in\cEnBN} \norm{\gN - \Pi_\gamma \gN}_\gamma^2 \right].
\end{equation}
\end{theorem}
\begin{proof}
Let $\beps_\bn = \bsigma^{(1)}_\bn - \theta_\bn \bnabla u_h$
and $K' \in \cT_\bn$ be an element.
By \eqref{eq:divsigma1nK}, the quantity $\beps_\bn$ satisfies
\begin{multline}
\label{eq:diveps}
    -\ddiv\beps_\bn 
  = \Pi_{K'} \theta_\bn (\Pi_{K'} f - \kappa_{K'}^2 u_h) - \bnabla \theta_\bn \cdot \bnabla u_h
    + \ddiv (\theta_\bn \bnabla u_h)
\\    
  = \Pi_{K'} \theta_\bn (\Pi_{K'} f - \kappa_{K'}^2 u_h + \Delta u_h)
  = \Pi_{K'} \theta_\bn r_h
  \quad \text{in } K',
\end{multline}
where $r_h = \Pi_{K'} f - \kappa_{K'}^2 u_h + \Delta u_h$ was introduced above \eqref{eq:rKest}.
Consequently, 
\begin{equation}
  \label{eq:divepsest}
  \norm{ \ddiv\beps_\bn }_{K'} = \norm{ \Pi_{K'} \theta_\bn r_h }_{K'} \leq \norm{ \theta_\bn r_h }_{K'} \leq \norm{r_h}_{K'}.
\end{equation}

Now, boundary conditions \eqref{eq:sigma1nKnuK} imply
\begin{equation}
\label{eq:epsnuK}
\beps_\bn \cdot \bnu_{K'} = \Pi_\gamma(\theta_\bn g_{K'}) - \theta_\bn \bnabla u_h \cdot \bnu_{K'}
    = \Pi_\gamma (\theta_\bn R)
    \quad\text{on all facets }\gamma\subset \partial K'
\end{equation}
and
\begin{equation}
\label{eq:epsnuKest}
  \norm{ \beps_\bn \cdot \bnu_{K'}}_{\partial K'}
 = \norm{\PiN (\theta_\bn R) }_{\partial K'}
 \leq \norm{ \theta_\bn R }_{\partial K'}
 \leq \norm{ R }_{\partial K' \setminus \gamma_\bn},
\end{equation}
where $\gamma_\bn$ stands for the facet of $K'$ opposite to the vertex $\bn$.
Since quantity $(\norm{\ddiv \bw}_{K'}^2 + \norm{\bw\cdot\bnu_{K'}}_{\partial K'}^2)^{1/2}$ is a norm in the finite dimensional space $\mathbf{RT}_1(K')$, we can use the scaling argument
\begin{equation*}
  \norm{ \bw }_{K'}^2 \leq C \left[ h_{K'}^2 \norm{\ddiv \bw}_{K'}^2 + h_{K'} \norm{\bw\cdot\bnu_{K'}}_{\partial K'}^2 \right]
  \quad \forall \bw \in \mathbf{RT}_1(K').
\end{equation*}
Thus, using $\bw = \beps_\bn|_{K'}$, inequalities \eqref{eq:divepsest} and \eqref{eq:epsnuKest}, we obtain
\begin{equation}
  \label{eq:epsn}
 \norm{ \beps_\bn }_{K'}^2 
 \leq C \left[ h_{K'}^2 \norm{ r_h }_{K'}^2 + h_{K'} \norm{R}_{\partial K'\setminus\gamma_\bn}^2 \right].
\end{equation}
Hence, estimates \eqref{eq:rKest} and \eqref{eq:Rest} applied in \eqref{eq:epsn} yield
$$
  \norm{ \beps_\bn }_{K'}^2 
  \leq C \left[
    \trinorm{u - u_h}_{\widetilde K'} + h_{K'}^2 \norm{f - \Pi f}_{\widetilde K'} 
      + h_{K'} \norm{\gN - \PiN \gN}_{\GammaN \cap \partial K'}
  \right].
$$
Finally, the bound \eqref{eq:etaKsigma_eff} follows by using the local quasi-uniformity of the mesh.
\end{proof}

\section{The Second Auxiliary Flux Reconstruction}
\label{se:sig2}

For elements $K\in\cT_h$, where $h_K \kappa_K > 1$, we define the second auxiliary flux reconstruction as
\begin{equation}
\label{eq:sigma2}
  \bsigma^{(2)}_{\widetilde K} = \sum_{\bn \in \cN_K} \bsigma_\bn^{(2)},
\end{equation}
where
\begin{equation}
\label{eq:sigma1n2}
  \bsigma_\bn^{(2)} = \kappa_K^{-2} \theta_\bn \bnabla f_\bn
\end{equation}
and $f_\bn$ is the $L^2(\omega_\bn)$-orthogonal projection of $f$ onto the space $\mathbb{P}_1(\omega_\bn)$ of affine functions on $\omega_\bn$. Note that $\bsigma_\bn^{(2)}$ is supported in $\omega_\bn$ and that it is continuous.
To simplify the notation we introduce piecewise linear and discontinuous function $f_\bn^\kappa$ in the patch $\omega_\bn$ by the rule
$$
  f_\bn^\kappa|_{K'} = \frac{\kappa_{K'}^2}{\kappa_K^2} f_\bn|_{K'}
  \quad \forall K' \in \cT_\bn.
$$
For completeness, we define $f_\bn^\kappa = \Pi f$ in those patches $\omega_\bn$, where $h_K\kappa_K \leq 1$ for all $K\in\cT_\bn$.


It is clear that $\bsigma^{(2)}_\bn|_K \in \mathbf{RT}_1(K)$ for all $K \in \cT_\bn$ and that $\bsigma^{(2)}_\bn \cdot \bnu_K = 0$ on facets $\gamma \in \cEnBE$. Thus $\bsigma^{(2)}_\bn$ can be extended by zero such that
$\bsigma^{(2)}_\bn \in \Hdiv$ 
and consequently $\bsigma^{(2)}_{\widetilde K} \in \Hdiv$.

\begin{lemma}
Let $u_h \in V_h$ be arbitrary.
Let $K\in\cT_h$ be an element such that $h_K \kappa_K > 1$.
Let $\bn\in\cN_K$ be its vertex and 
let $\bsigma^{(2)}_\bn$ be defined by \eqref{eq:sigma1n2}.
Then
\begin{multline}
  \label{eq:bsigman2est}
  \norm{ \bsigma^{(2)}_\bn - \theta_\bn \bnabla u_h }_{K'}^2 
  + \kappa_{K'}^{-2} \norm{ \theta_\bn (\Pi_{K'} f - \kappa_{K'}^2 u_h) 
   - \bnabla \theta_\bn \cdot \bnabla u_h + \ddiv \bsigma^{(2)}_\bn }_{K'}^2
\\   
  \leq C \left( \trinorm{u - u_h}_{K'}^2 + \kappa_{K'}^{-2} \norm{f - \Pi_{K'} f}_{K'}^2 + \kappa_{K'}^{-2} \norm{
  f_\bn^\kappa 
  - \Pi_{K'} f}_{K'}^2 \right)
\end{multline}
holds for all $K' \in \cT_\bn$.
\end{lemma}
\begin{proof}
Let $K' \in \cT_\bn$ be fixed.
Recall that Lemma~\ref{le:kappa} implies $\kappa_{K'} > 0$, $\kappa_K \leq C \kappa_{K'}$ and $\kappa_{K'}^{-1} h_{K'}^{-1} \leq C$.  
The inverse inequality yields
\begin{multline}
\label{eq:estsig2}
  \norm{ \bsigma^{(2)}_\bn - \theta_\bn \bnabla u_h }_{K'}
  = \norm{\theta_\bn \bnabla( \kappa_K^{-2} f_\bn - u_h) }_{K'}
  \leq \norm{\bnabla( \kappa_K^{-2} f_\bn - u_h) }_{K'}
\\  
  \leq C h_{K'}^{-1} \kappa_{K'}^{-2} \norm{\frac{\kappa_{K'}^2}{\kappa_K^2} f_\bn - \kappa_{K'}^2 u_h }_{K'}
  \leq C \kappa_{K'}^{-1} \norm{f_\bn^\kappa - \kappa_{K'}^2 u_h }_{K'}.
\end{multline}
Similarly, 
\begin{align*}
  &\norm{ \theta_\bn (\Pi_{K'} f - \kappa_{K'}^2 u_h) 
     - \bnabla \theta_\bn \cdot \bnabla u_h + \ddiv \bsigma^{(2)}_\bn }_{K'}
\\      
  &\quad\leq  \norm{ \theta_\bn (\Pi_{K'} f - \kappa_{K'}^2 u_h) }_{K'}
      + \norm{ \ddiv (\bsigma^{(2)}_\bn - \theta_\bn \bnabla u_h) }_{K'}
\\      
  &\quad\leq  \norm{ \Pi_{K'} f - \kappa_{K'}^2 u_h }_{K'}
      + C h_{K'}^{-1} \norm{ \bsigma^{(2)}_\bn - \theta_\bn \bnabla u_h }_{K'}
\\      
  &\quad\leq  \norm{ \Pi_{K'} f - \kappa_{K'}^2 u_h }_{K'}
      + C \kappa_{K'}^{-1} h_{K'}^{-1} \norm{ f_\bn^\kappa - \kappa_{K'}^2 u_h }_{K'},
\end{align*}
where estimate \eqref{eq:estsig2} and the fact that $\Delta u_h|_{K'} = 0$ were used.
Consequently, using bound $\kappa_{K'}^{-1} h_{K'}^{-1} \leq C$ and triangle inequality, we obtain
\begin{align*}
&\norm{ \bsigma^{(2)}_\bn - \theta_\bn \bnabla u_h }_{K'}^2 
  + \kappa_{K'}^{-2} \norm{ \theta_\bn (\Pi_{K'} f - \kappa_{K'}^2 u_h) 
   - \bnabla \theta_\bn \cdot \bnabla u_h + \ddiv \bsigma^{(2)}_\bn }_{K'}^2
\\   
&\quad\leq C \kappa_{K'}^{-2} \left(
    \norm{ f_\bn^\kappa - \kappa_{K'}^2 u_h }_{K'}^2 + \norm{ \Pi_{K'} f - \kappa_{K'}^2 u_h }_{K'}^2
  \right)
\\
&\quad\leq C \kappa_{K'}^{-2} \left(
    \norm{ f_\bn^\kappa - \Pi_{K'} f }_{K'}^2 + \norm{ \Pi_{K'} f - \kappa_{K'}^2 u_h }_{K'}^2
  \right)  
\\  
&\quad\leq C \left(  
    \trinorm{u - u_h}_{K'}^2 
  + \kappa_{K'}^{-2} \norm{f - \Pi_{K'} f}_{K'}^2
  + \kappa_{K'}^{-2} \norm{ f_\bn^\kappa - \Pi_{K'} f}_{K'}^2
  \right),  
\end{align*}
where we employ \eqref{eq:rKest}. 
\end{proof}

\begin{lemma}
Let $u_h \in V_h$ be arbitrary.
Let $K\in\cT_h$ be an element such that $\kappa_K h_K > 1$.
Let $\bn\in\cN_K$ be its vertex and 
let $\gamma \in \cEnBN$ be a facet on $\GammaN$ adjacent to an element $K_\gamma$.
Then 
\begin{multline}
  \label{eq:PigNsig2est}
  \CT^{K,\gamma} \norm{\Pi_\gamma(\theta_\bn \Pi_\gamma \gN) - \bsigma^{(2)}_\bn \cdot\bnu }_\gamma
  \leq 
\\  
  C \left( \trinorm{u-u_h}_{K_\gamma} 
    + \kappa_{K_\gamma}^{-1} \norm{f - \Pi_{K_\gamma} f}_{K_\gamma}
    + \kappa_{K_\gamma}^{-1} \norm{f - f_\bn^\kappa}_{K_\gamma}  
    + \kappa_{K_\gamma}^{-1/2}\norm{\gN - \Pi_\gamma\gN}_\gamma  
  \right).
\end{multline}
\end{lemma}
\begin{proof}
Using the definition \eqref{eq:sigma1n2} of $\bsigma^{(2)}_\bn$, properties of projection $\Pi_\gamma$ and hat function $\theta_\bn$, we obtain
\begin{multline}
  \label{eq:estPigN}
       \norm{\Pi_\gamma(\theta_\bn \Pi_\gamma \gN) - \bsigma^{(2)}_\bn \cdot\bnu }_\gamma
  =    \norm{\Pi_\gamma\left(\theta_\bn \Pi_\gamma \gN - \kappa_K^{-2} \theta_\bn \bnabla f_\bn \cdot\bnu\right) }_\gamma
\\  
  \leq \norm{\theta_\bn \left( \Pi_\gamma \gN - \kappa_K^{-2} \bnabla f_\bn \cdot\bnu \right) }_\gamma
  \leq \norm{\Pi_\gamma \gN - \kappa_K^{-2} \bnabla f_\bn \cdot\bnu }_\gamma.
\end{multline}
To bound this norm, we consider a special test function. We define function $v$ on the boundary $\partial {K_\gamma}$ as $v = 0$ on $\partial {K_\gamma} \setminus \gamma$ and $v = \beta \cdot (\Pi_\gamma\gN - \kappa_K^{-2} \bnabla f_\bn \cdot \bnu)$ on $\gamma$, where $\beta = \prod_{\bn\in\cN_\gamma} \theta_\bn$ is a bubble function defined on the facet $\gamma$. Then we introduce minimum energy extension $\cE v$ to the interior of ${K_\gamma}$ satisfying $\cE v \in H^1({K_\gamma})$, $\cE v = v$ on $\partial {K_\gamma}$, and $\cB_{K_\gamma}(\cE v, w) = 0$ for all $w \in H^1_0({K_\gamma})$, see \cite[Section~3.1]{AinBab:1999}. Extending $\cE v$ further by zero to the rest of the domain $\Omega$, we have $\cE v \in V$.

Using $\cE v$ as a test function in \eqref{eq:weakf} together with identity $\int_\gamma \bnabla f_\bn \cdot \bnu v \ds = \int_{K_\gamma} \bnabla f_\bn \cdot \bnabla \cE v \dx$, we derive
\begin{multline}
\label{eq:eqgN}
\int_\gamma ( \Pi_\gamma \gN - \kappa_K^{-2} \bnabla f_\bn \cdot \bnu ) v \ds
\\  
  =   \int_{K_\gamma} \left( \bnabla u \cdot \bnabla \cE v
    + \kappa_{K_\gamma}^2 u \cE v
    - f \cE v \right) \dx
    - \int_\gamma \kappa_K^{-2} \bnabla f_\bn \cdot \bnu v \ds
    + \int_\gamma (\Pi_\gamma\gN - \gN) v \ds
\\    
  = \int_{K_\gamma} (\bnabla u - \bnabla u_h)\cdot \bnabla \cE v \dx
  + \kappa_{K_\gamma}^2 \int_{K_\gamma} (u-u_h) \cE v \dx
\\
  +  \int_{K_\gamma} (  \bnabla u_h - \kappa_K^{-2} \bnabla f_\bn) \cdot \bnabla \cE v \dx
  + \int_{K_\gamma} (\kappa_{K_\gamma}^2 u_h - f) \cE v \dx
  + \int_\gamma (\Pi_\gamma\gN - \gN) v \ds.
\end{multline}
Since $\Pi_\gamma\gN - \kappa_K^{-2} \bnabla f_\bn \cdot \bnu \in \mathbb{P}_1(\gamma)$ and $\left(\int_\gamma \beta \varphi^2 \dx \right)^{1/2}$ is a norm in $\mathbb{P}_1(\gamma)$, we use the equivalence of norms in the finite dimensional space $\mathbb{P}_1(\gamma)$ to get
$$
\norm{\Pi_\gamma \gN - \kappa_K^{-2} \bnabla f_\bn \cdot \bnu}_\gamma^2 
 \leq C \int_\gamma ( \Pi_\gamma \gN - \kappa_K^{-2} \bnabla f_\bn \cdot \bnu ) v \ds.
$$
This estimate together with identity \eqref{eq:eqgN}, Cauchy--Schwarz inequality, inverse inequality, and bound $h_{K_\gamma}^{-1}\kappa_{K_\gamma}^{-1} \leq C$ provided by assumption \eqref{eq:kappacond1} yields
\begin{multline}
  \label{eq:estgN}
  \norm{\Pi_\gamma \gN - \kappa_K^{-2} \bnabla f_\bn \cdot \bnu}_\gamma^2
  \leq C \left( \trinorm{u-u_h}_{K_\gamma} \trinorm{\cE v}_{K_\gamma}
    + h_{K_\gamma}^{-1} \norm{ u_h - \kappa_K^{-2}f_\bn}_{K_\gamma} \norm{\bnabla \cE v}_{K_\gamma}
    \right.
\\    
   \left.
    + \norm{\kappa_{K_\gamma}^2 u_h - f}_{K_\gamma} \norm{\cE v}_{K_\gamma}
    + \norm{\Pi_\gamma\gN - \gN}_\gamma \norm{v}_\gamma
   \right) 
\\    
  \leq C \left( \trinorm{u-u_h}_{K_\gamma} 
    + \kappa_{K_\gamma}^{-1} \norm{\kappa_{K_\gamma}^2 u_h - f}_{K_\gamma} 
             + \kappa_{K_\gamma}^{-1} \norm{f - \frac{\kappa_{K_\gamma}^2}{\kappa_K^2} f_\bn}_{K_\gamma} \right) \trinorm{\cE v}_{K_\gamma}
\\             
    + C \norm{\gN - \Pi_\gamma\gN}_\gamma \norm{v}_\gamma.
\end{multline}
%
Here, the energy norm $\trinorm{\cE v}_{K_\gamma}$ is bounded by \cite[Lemma~4]{AinBab:1999} as
$$
  \trinorm{\cE v}_{K_\gamma} \leq C \min\{h_{K_\gamma},\kappa_{K_\gamma}^{-1}\}^{-1/2} \norm{\Pi_\gamma\gN - \kappa_K^{-2} \bnabla f_\bn \cdot \bnu}_\gamma 
$$
and norm $\norm{v}_\gamma$ by using the equivalence of norms in $\mathbb{P}_1(\gamma)$ as
$$
  \norm{v}_\gamma \leq C \norm{\Pi_\gamma \gN - \kappa_K^{-2} \bnabla f_\bn \cdot \bnu}_\gamma.
$$
Finally, using these bounds, inequalities $h_{K_\gamma}^{-1}\kappa_{K_\gamma}^{-1} \leq C$ and $\CT^{K,\gamma} \leq C \kappa_{K_\gamma}^{-1/2}$, and estimate \eqref{eq:rKest} in \eqref{eq:estgN}, we obtain
\begin{multline*}
  \CT^{K,\gamma} \norm{\Pi_\gamma \gN - \kappa_K^{-2} \bnabla f_\bn \cdot \bnu}_\gamma
\\  
  \leq C \left[ \trinorm{u-u_h}_{K_\gamma} 
    + \kappa_{K_\gamma}^{-1} \norm{f - \Pi_{K_\gamma} f}_{K_\gamma}
\right.
\\  
\left.    
    + \kappa_{K_\gamma}^{-1} \norm{f - f_\bn^\kappa}_{K_\gamma}  
    + \kappa_{K_\gamma}^{-1/2}\norm{\gN - \Pi_\gamma\gN}_\gamma  
  \right].
\end{multline*}
This estimate and \eqref{eq:estPigN} finish the proof.
\end{proof}

\section{Efficiency of Patchwise Minimizations}
\label{se:efficiency}


We first formulate a lemma stating that error indicators can be bounded by the value of the quadratic functional $E_\bn$.
\begin{lemma}
Let $u_h \in V$ be arbitrary. 
Let $d$ stand for the dimension.
Let error indicators $\eta_K$ be defined by \eqref{eq:etaK}.
Let reconstructed flux $\btau\in\Hdiv$ given by \eqref{eq:tau} 
satisfy equilibration conditions \eqref{eq:equilib1}--\eqref{eq:equilib2}
and 
let its local components $\btau_\bn$ be in $\bW(\omega_\bn)$.
Then
\begin{equation}
  \label{eq:estetaK}
  \eta_K^2(\btau) \leq (d+2)(d+1) \sum_{\bn\in\cN_K} E_\bn(\btau_\bn)
  \quad\text{for all } K \in \cT_h.
\end{equation}
\end{lemma}
\begin{proof}
Let simplex $K\in\cT_h^+$ be fixed.
Since there is $d+1$ facets on the boundary of $K$, we can bound $\eta_K(\btau)$ by Cauchy--Schwarz  inequality as
$$
  \eta_K^2(\btau) \leq (d+2) \left( 
      \norm{ \beps }_K^2 
    + \norm{ \kappa^{-1} r }_K^2
    + \sum_{\gamma \subset \GammaN\cap\partial K} \norm{ \CT^{K,\gamma} \RN}_\gamma^2
  \right).
$$

Using the partition of unity $\theta_\bn$,
definition of $\beps$, definition \eqref{eq:tau} of $\btau$, and Cauchy--Schwarz inequality, we obtain
\begin{equation}
  \label{eq:bepsKest}
  \norm{ \beps }_K^2 = \norm{ \sum_{\bn\in\cN_K} \left( \btau_\bn - \theta_\bn \bnabla u_h \right) }_K^2
  \leq (d+1) \sum_{\bn\in\cN_K} \norm{ \btau_\bn - \theta_\bn \bnabla u_h  }_{\omega_\bn}^2.
\end{equation}
Similarly, we estimate
\begin{multline*}
  \norm{ \kappa^{-1} r }_K^2
  = \norm{ \kappa^{-1} \sum_{\bn\in\cN_K} \left[ \Pi_K( \theta_\bn(\Pi_K f - \kappa_K^2 u_h)) + \ddiv \btau_\bn \right] }_K^2
\\  
  \leq (d+1) \sum_{\bn\in\cN_K} \norm{ \kappa^{-1} \left[ \Pi( \theta_\bn(\Pi f - \kappa^2 u_h)) + \ddiv \btau_\bn \right] }_{\omega_\bn^+}^2
\end{multline*}
and for facets $\gamma \subset \GammaN \cap \partial K$
\begin{multline*}
    \norm{\CT^{K,\gamma} \RN}_\gamma^2
  = \norm{\CT^{K,\gamma} \sum_{\bn\in\cN_\gamma} \left( \Pi_\gamma (\theta_\bn \Pi_\gamma \gN) - \btau_\bn \cdot \bnu \right) }_\gamma^2
\\  
  \leq d \sum_{\bn\in\cN_\gamma} \norm{ \CT^{K,\gamma} \left[ \Pi_\gamma (\theta_\bn \Pi_\gamma \gN) - \btau_\bn \cdot \bnu \right] }_\gamma^2.
\end{multline*}
Statement \eqref{eq:estetaK} now follows by a combination of these estimates.

If $K\in\cT_h^0$ then bound \eqref{eq:estetaK} is easy to verify, because of
identity $\eta_K(\btau) = \norm{ \beps }_K$ and estimate \eqref{eq:bepsKest}.
\end{proof}

Note that this lemma holds true for any flux $\btau\in\Hdiv$ given by \eqref{eq:tau}. Its local components $\btau_\bn \in \bW(\omega_\bn)$ are not required to minimize the quadratic functional $E_\bn$ defined in \eqref{eq:taunmin}.

The following theorem presents the main result of this paper. It states the efficiency and robustness of error indicators computed by \eqref{eq:etaK} from the patchwise flux reconstruction $\btau \in \Hdiv$ defined in \eqref{eq:tau}.
To formulate it, we introduce the union $\tGammaNK$ of facets in the triangulation $\cT_h$ that lie on $\GammaN$ and have at least one common point with $K$, i.e., $\tGammaNK = \bigcup \{ \gamma \subset \GammaN : \gamma \cap K \neq \emptyset\}$.
\begin{theorem}
\label{th:loceff}
Let $u\in V$ be the weak solution \eqref{eq:weakf} and 
let $u_h \in V_h$ be its Galerkin approximation satisfying \eqref{eq:FEM}.
Let flux reconstruction $\btau \in \Hdiv$ be given by \eqref{eq:tau} and 
let its local components $\btau_\bn \in \bW(\omega_\bn)$ solve local problems \eqref{eq:taun}--\eqref{eq:taunconstr2}.
Then there exists a constant $C > 0$ independent of reaction coefficient $\kappa$ and any mesh size 
such that the local efficiency estimate 
\begin{multline}
  \label{eq:loceff}
  \eta_K^2(\btau) \leq C \left[ 
       \trinorm{u - u_h}_{\widetilde{\widetilde K}}^2 
       + \min\{h_K,\kappa_K^{-1}\}^2 \left( \norm{f - \Pi f}_{\widetilde{\widetilde K}}^2
       + \sum_{\bn\in\cN_K} \norm{ f_\bn^\kappa - \Pi f}_{\omega_\bn}^2 \right)
\right. 
\\       
\left.
       + \min\{h_K,\kappa_K^{-1}\} \norm{\gN - \PiN \gN}_{\tGammaNK}^2 
  \right].
\end{multline}
holds true for all elements $K\in\cT_h$.  

\end{theorem}
\begin{proof}
We consider two cases. 
First, let $K \in \cT_h$ be such that $\kappa_K h_K \leq 1$. 
Since 
$\btau_\bn \in \bW(\omega_\bn)$ minimizes the functional $E_\bn$,
both $\btau_\bn$ and $\bsigma^{(1)}_\bn$ satisfy constraints \eqref{eq:tauconstr1}--\eqref{eq:tauconstr2},
and fluxes $\bsigma^{(1)}_\bn$ satisfy \eqref{eq:Ensig1},
we obtain from \eqref{eq:estetaK} the estimate
$$
  \eta_K^2(\btau) \leq (d+2)(d+1) \sum_{\bn\in\cN_K} E_\bn(\bsigma^{(1)}_\bn)
  = C \sum_{\bn\in\cN_K} \norm{ \bsigma^{(1)}_\bn - \theta_\bn \bnabla u_h }_{\omega_\bn}^2.
$$
Applying Theorem~\ref{th:sigma} in this inequality immediately yields
\begin{equation}
  \label{eq:etaeff1}
  \eta_K^2(\btau) \leq 
  C \left[ \trinorm{u-u_h}_{\widetilde{\widetilde K}}^2 
    + h_K^2 \norm{f - \Pi f}_{\widetilde{\widetilde K}}^2 
    + h_K \sum_{\bn\in\cN_K} \sum_{\gamma \in \cEnBN} \norm{\gN - \Pi_\gamma \gN}_\gamma^2 \right].
\end{equation}

Second, let $K \in \cT_h$ be such that $\kappa_K h_K > 1$. 
By assumption \eqref{eq:kappacond1}, the reaction coefficient satisfies $\kappa_{K'} > 0$ for all elements $K' \subset \widetilde K$. 
Since $\bsigma^{(2)}_\bn \in \bW(\omega_\bn)$
and $\btau_\bn$ is the minimizer of $E_\bn$, the inequality \eqref{eq:estetaK}
yields
\begin{multline*}
  \eta_K^2(\btau) \leq (d+2)(d+1) E_\bn(\bsigma^{(2)}_\bn)
    \leq C \sum_{\bn\in\cN_K} \left[ 
        \norm{ \bsigma^{(2)}_\bn - \theta_\bn \bnabla u_h }_{\omega_\bn}^2
\right.
\\
        + \sum_{K'\in\cT_\bn} \kappa_{K'}^{-2} \norm{\Pi_{K'}(\theta_\bn (\Pi_{K'} f - \kappa_{K'}^2 u_h))
          - \bnabla \theta_\bn \cdot \bnabla u_h + \ddiv\bsigma^{(2)}_\bn}_{K'}^2
\\
\left.
  + \sum_{\gamma \in \cEnBN} \left(\CT^{K_\gamma,\gamma}\right)^2 \norm{\Pi_\gamma(\theta_\bn \Pi_\gamma \gN) - \bsigma^{(2)}_\bn \cdot\bnu }_\gamma^2
       \right].
\end{multline*}
Estimates \eqref{eq:bsigman2est} and \eqref{eq:PigNsig2est} then give 
\begin{multline}
  \label{eq:etaeff2}
  \eta_K^2(\btau) 
\leq C \sum_{\bn\in\cN_K} \left[ 
    \sum_{K'\in\cT_\bn} 
      \left( \trinorm{u - u_h}_{K'}^2 + \kappa_{K'}^{-2} \norm{f - \Pi_{K'} f}_{K'}^2 
\right.\right.
\\
\left.\left.
      + \kappa_{K'}^{-2} \norm{ f_\bn^\kappa - \Pi_{K'} f}_{K'}^2 \right)
      +
     \sum_{\gamma \in \cEnBN} \kappa_{K_\gamma}^{-1} \norm{\gN - \Pi_\gamma\gN}_\gamma^2
     \right].
\end{multline}
Assumption \eqref{eq:kappacond1} and a combination of \eqref{eq:etaeff1} and \eqref{eq:etaeff2} finishes the proof.
%
%
\end{proof}

The oscillation term $\min\{h_K,\kappa_K^{-1}\} \norm{ f_\bn^\kappa - \Pi f}_{\omega_\bn}$ is not standard, however, as well as the other oscillation terms in \eqref{eq:loceff} it is of higher order than the error $\trinorm{u - u_h}_{\widetilde{\widetilde K}}$ and does not spoil the robust local efficiency result.

\section{Avoiding Equilibration}
\label{se:noconstr}

Theorem~\ref{th:main} requires the flux $\btau$ to satisfy equilibration conditions
\eqref{eq:equilib1}--\eqref{eq:equilib2}. 
Therefore, the local minimization problems \eqref{eq:taunmin} is constrained by \eqref{eq:tauconstr1}--\eqref{eq:tauconstr2}.
However, in practical computations constraints \eqref{eq:tauconstr1}--\eqref{eq:tauconstr2} 
are often not satisfied exactly due to round-off errors. Consequently, assumptions of Theorem~\ref{th:main}
are not valid and error estimator $\eta(\btau)$ is not guaranteed to provide 
the upper bound on the error.
This problem can be avoided by introducing Friedrichs--Poincar\'e and trace inequalities and two additional parameters in the definition of the estimator.

Given arbitrary flux $\btau\in\Hdiv$, we define modified local error indicators
\begin{equation}
\label{eq:etaKB}
  \teta_K(\btau) = \left\{ 
    \begin{array}{ll}
      \left( \norm{\beps}_K^2
      + \kappa_K^{-2} \norm{r}_K^2 \right)^{1/2} 
      + \sum_{\gamma \subset \GammaN\cap\partial K} \CT^{K,\gamma} \norm{\RN}_\gamma
      & \text{if } K \in \cT_h^+,    
    \\ 
      \left( \norm{\beps}_K^2 + \kappa_0^{-2} \norm{r}_K^2  
      + \zeta_0^{-2} \sum_{\gamma \subset \GammaN\cap\partial K} \norm{\RN}_\gamma^2
      \right)^{1/2}
      & \text{if } K \in \cT_h^0,
     \end{array}
   \right.
\end{equation}
where small parameter $\kappa_0 > 0$ replaces the zero value of $\kappa$ in a sense
and small parameter $\zeta_0 > 0$ has a similar meaning.
Notice that $\teta_K(\btau)$ differs from $\eta_K(\btau)$ only for elements $K\in\cT_h^0$.

The modified error estimator is then defined as 
\begin{equation}
\label{eq:etaB}
  \teta^2(\btau) = 
     \left(1+\kappa_0^2\CF^2 + \zeta_0^2 \CT^2\right) 
     \sum_{K\in\cT_h} [\teta_K(\btau) + \osc_K(f,\gN)]^2,
\end{equation}
where $\CF > 0$ and $\CT > 0$ are constants from
Friedrichs--Poincar\'e and trace inequalities
\begin{equation}
  \label{eq:Pointrace}
  \norm{v} \leq \CF \trinorm{v}
  \quad\text{and}\quad
  \norm{v}_\GammaN \leq \CT \trinorm{v}
  \quad\forall v \in V.
\end{equation}
The following theorem presents a modification of Theorem~\ref{th:main} that avoids the equilibration conditions \eqref{eq:equilib1}--\eqref{eq:equilib2}.
\begin{theorem}
\label{th:mainB}
Let $u_h \in V$ and $\btau \in\Hdiv$ be arbitrary. Then
\begin{equation}
\label{eq:upperboundB}
\trinorm{u - u_h} \leq 
  \teta(\btau)
\end{equation}
for all $\kappa_0 > 0$ and $\zeta_0 > 0$.
\end{theorem}
\begin{proof}
Using identity \eqref{eq:errid} and estimates \eqref{eq:oscKest}--\eqref{eq:RNest},
we arrive at
\begin{multline*}
  \cB(u - u_h,v)
  \leq \sum\limits_{K\in\cT_h^+} \left[
    \left( \norm{\beps}_K^2 + \kappa_K^{-2}\norm{r}_K^2 \right)^{1/2}
    + \sum_{\gamma\subset\GammaN\cap\partial K} \CT^{K,\gamma} \norm{\RN}_\gamma
  \right] \trinorm{v}_K
\\
  + \sum\limits_{K\in\cT_h^0} \left[
   \norm{\beps}_K \trinorm{v}_K
   + \norm{r}_K \norm{v}_K
   + \sum\limits_{\gamma\subset\GammaN\cap\partial K} \norm{\RN}_\gamma \norm{v}_\gamma
   \right]
\\
  + \sum\limits_{K\in\cT_h} \osc_K(f,\gN) \trinorm{v}_K.
\end{multline*}
Using Cauchy--Schwarz inequality 
\begin{multline*}
   \norm{\beps}_K \trinorm{v}_K
   + \norm{r}_K \norm{v}_K
   + \sum\limits_{\gamma\subset\GammaN\cap\partial K} \norm{\RN}_\gamma \norm{v}_\gamma
\\   
\leq
\left( 
   \norm{\beps}_K^2 
   + \kappa_0^{-2} \norm{r}_K^2 
   + \zeta_0^{-2} \sum\limits_{\gamma\subset\GammaN\cap\partial K} \norm{\RN}_\gamma^2 
   \right)^{1/2}
\left( 
   \trinorm{v}_K^2 + \kappa_0^2 \norm{v}_K^2 
   + \zeta_0^2 \norm{v}_{\GammaN\cap\partial K}^2  
   \right)^{1/2},
\end{multline*}
we obtain
\begin{multline*}
  \cB(u - u_h,v) 
  \leq 
    \sum_{K\in\cT_h} \left[ \teta_K(\btau) + \osc_K(f,\gN) \right]
      \left( 
        \trinorm{v}_K^2 + \kappa_0^2 \norm{v}_K^2 
        + \zeta_0^2 \norm{v}_{\GammaN\cap\partial K}^2  
      \right)^{1/2}
\\
  \leq
  \left(\sum_{K\in\cT_h} \left[ \teta_K(\btau) + \osc_K(f,\gN) \right]^2\right)^{1/2}
  \left( \trinorm{v}^2 + \kappa_0^2 \norm{v}^2 
        + \zeta_0^2 \norm{v}_\GammaN^2  \right)^{1/2}.
\end{multline*}
Friedrichs--Poincar\'e and trace inequalities \eqref{eq:Pointrace},
notation \eqref{eq:etaB}, and choice $v=u-u_h$ finish the proof.
\end{proof}

Constants $\kappa_0$ and $\zeta_0$ should be small. Ideally so small that $1 +
\kappa_0^2 \CF^2 + \zeta_0^2 \CT^2 \approx 1$.  In this case the influence of
Friedrichs--Poincar\'e and trace constants $\CF$ and $\CT$ on the value of the
error bound \eqref{eq:upperboundB} is negligible.  In the case of pure Dirichlet
boundary conditions, i.e., $\GammaN = \emptyset$, the parameter $\zeta_0$ is
not needed and estimate \eqref{eq:upperboundB} holds with $\CT=0$.  In case
$\kappa > 0$ everywhere in $\Omega$ the set $\cT_h^0$ is empty and parameters
$\kappa_0$, $\zeta_0$ and constants $\CF$, $\CT$ are not needed.
Estimate \eqref{eq:upperboundB} then holds with $\CF=\CT=0$
and local error indicators \eqref{eq:etaK} and \eqref{eq:etaKB} coincide.
However, if $\kappa$ vanishes at some parts of $\Omega$ and
constants $\CF$ and $\CT$ are needed, then they can be computed analytically in
some special cases and numerically, in general. Even their guaranteed numerical
bounds are available, see e.g. \cite{SebVej:2014,Vejchodsky:2018}.

Since Theorem~\ref{th:mainB} does not require any equilibration condition, 
the patchwise flux reconstruction procedure simplifies. 
Modified fluxes $\tbtau_\bn \in \bW(\omega_\bn)$ minimize the quadratic functional
\begin{multline}
  \label{eq:taunminB}
  \widetilde{E}_\bn(\tbtau_\bn) = \norm{ \tbtau_\bn - \theta_\bn \bnabla u_h }_{\omega_\bn}^2 \\   
  + \norm{ \tilde\kappa^{-1} \left[ \Pi(\theta_\bn (\Pi f - \kappa^2 u_h))
  - \bnabla \theta_\bn \cdot \bnabla u_h + \ddiv \tbtau_\bn \right] }_{\omega_\bn}^2 \\
  + \norm{ \tilde\zeta^{-1} \left[ \PiN(\theta_\bn \PiN \gN) - \tbtau_\bn \cdot\bnu \right] }_\GammaNn^2, 
\end{multline}
over the space $\bW(\omega_\bn)$, where $\GammaNn$ is
the union of the facets belonging to $\cEnBN$ and
piecewise constant parameters $\tilde\kappa$ and $\tilde\zeta$ are given by
\begin{equation}
  \label{eq:beta}
  \tilde\kappa|_K = \left\{
  \begin{array}{ll}
    \kappa_K & \text{if }\kappa_K > 0, \\
    \kappa_0 & \text{if }\kappa_K = 0,
  \end{array}
  \right.
  \quad\text{and}\quad
  \tilde\zeta|_\gamma = \left\{
  \begin{array}{ll}
    \left( \CT^{K_\gamma,\gamma} \right)^{-1}  & \text{if }\kappa_{K_\gamma} > 0, \\
    \zeta_0 & \text{if }\kappa_{K_\gamma} = 0,
  \end{array}
  \right.
\end{equation}
for all elements $K\in\cT_h$ and all facets $\gamma \subset \GammaN$,
where we recall that $K_\gamma$ denotes the element adjacent to the facet $\gamma$.

The minimizer $\tbtau_\bn \in \bW(\omega_\bn)$ of \eqref{eq:taunminB} could,
equally well, be characterised as the unique solution of the following problem:
\begin{multline}
  \label{eq:taunB}
  (\tilde\kappa^{-2} \ddiv \tbtau_\bn, \ddiv \bw_h)_{\omega_\bn}
  + (\tbtau_\bn, \bw_h)_{\omega_\bn} 
  + (\tilde\zeta^{-2} \tbtau_\bn\cdot\bnu,\bw_h\cdot\bnu)_\GammaNn
\\
  = (\theta_\bn \bnabla u_h, \bw_h)_{\omega_\bn} 
  - \left( \tilde\kappa^{-2}\left[ \theta_\bn (\Pi f - \kappa^2 u_h) - \bnabla \theta_\bn \cdot \bnabla u_h \right], \ddiv \bw_h \right)_{\omega_\bn}
\\  
  + (\tilde\zeta^{-2} \theta_\bn\PiN\gN,\bw_h\cdot\bnu)_\GammaNn
\end{multline}
for all $\bw_h \in \bW(\omega_\bn)$.
Note that the large values $\kappa_0^{-1}$ and $\zeta_0^{-1}$ play here the role of 
penalty parameters to impose constraints \eqref{eq:tauconstr1}--\eqref{eq:tauconstr2} 
and \eqref{eq:taunconstr1}--\eqref{eq:taunconstr2} in a weak sense.
Consequently, the difference between $\btau$ and $\tbtau$ is small in practical computations.

Summing up fluxes $\tbtau_\bn$ as in Section~\ref{se:locmin} results 
in a modified reconstructed flux
\begin{equation}
\label{eq:tauB}
\tbtau = \sum_{\bn \in \cN_h} \tbtau_\bn
\end{equation}
that can be directly used in Theorem~\ref{th:mainB} to obtain a guaranteed 
upper bound on the error.

The modified reconstructed flux is locally efficient and robust as 
it is stated in the following corollary.
\begin{corollary}
Let $u\in V$ be the weak solution \eqref{eq:weakf} and 
let $u_h \in V_h$ be its Galerkin approximation satisfying \eqref{eq:FEM}.
Let flux reconstruction $\tbtau \in \Hdiv$ be given by \eqref{eq:tauB} and 
let its local components $\tbtau_\bn \in \bW(\omega_\bn)$ solve local problems \eqref{eq:taunB}.
Then there exists a constant $C > 0$ independent of reaction coefficient $\kappa$ and any mesh size 
such that the local efficiency estimate 
\begin{multline*}
  \teta_K^2(\tbtau) \leq C \left[ 
       \trinorm{u - u_h}_{\widetilde{\widetilde K}}^2 
       + \min\{h_K,\kappa_K^{-1}\}^2 \left( \norm{f - \Pi f}_{\widetilde{\widetilde K}}^2
       + \sum_{\bn\in\cN_K} \norm{ f_\bn^\kappa - \Pi f}_{\omega_\bn}^2 \right)
\right. 
\\       
\left.
       + \min\{h_K,\kappa_K^{-1}\} \norm{\gN - \PiN \gN}_{\tGammaNK}^2 
  \right].
\end{multline*}
holds true for all elements $K\in\cT_h$.  
\end{corollary}
\begin{proof}
The proof follows the same lines as the proof of Theorem~\ref{th:loceff}.
In particular, we use the fact that
$$
  \teta_K^2(\tbtau) \leq (d+2)(d+1) \sum_{\bn\in\cN_K} \widetilde E_\bn(\tbtau_\bn)
  \quad\text{for all } K \in \cT_h.
$$
\end{proof}

\section{Numerical Examples}
\label{se:numex}


\paragraph*{Example 1}
In this example, we consider problem \eqref{eq:modpro} in a domain with reentrant corner:
$\Omega = \{ (\varrho,\phi) : 0 \leq \varrho < 1 \text{ and } \phi \in (\pi/2,2\pi)\}$, 
where $\varrho$ and $\phi$ are standard polar coordinates, see Figure~\ref{fi:mesh} (left).
The boundary conditions
are homogeneous Dirichlet only, i.e., $\GammaD = \partial\Omega$ and $\GammaN = \emptyset$.
The reaction coefficient $\kappa$ is assumed positive, constant in $\Omega$, and its specific values are provided below.
Choosing the right-hand side as
$ f= \kappa^2 \varrho^{2/3} \sin (2\phi-\pi)/3 $, 
the exact solution is explicitly given by
$$
  u = \left( \varrho^{2/3} - \frac{I_{2/3}(\kappa \varrho)}{I_{2/3}(\kappa)} \right) \sin \frac{2\phi-\pi}{3},
$$
where $I_\alpha$ stands for the modified Bessel function of the first kind.
This solution exhibits singularity at the origin and a boundary layer at $\varrho=1$ for large values of $\kappa$.

We first compute the finite element solution \eqref{eq:FEM} using the mesh shown in Figure~\ref{fi:mesh} (right) 
for $\kappa = 10^{-3}$, $10^{-2}, \dots, 10^6$. For each value of $\kappa$ we compute flux reconstruction 
\eqref{eq:tau} by solving local problems \eqref{eq:taun}--\eqref{eq:taunconstr2} and evaluate the error estimator $\eta(\btau)$ given by \eqref{eq:eta}.
Note that since $\GammaN = \emptyset$ and $\kappa > 0$, the procedure considerably simplifies. 
The set $\cT_h^0$ is empty, equilibration conditions \eqref{eq:equilib1}--\eqref{eq:equilib2} do not apply
as well as constraints \eqref{eq:taunconstr1}--\eqref{eq:taunconstr2}.
Reconstructed fluxes $\btau$ and $\tbtau$ given by \eqref{eq:tau} and \eqref{eq:tauB}, respectively,
are identical and $\eta_K(\btau) = \teta_K(\tbtau)$ for all $K \in \cT_h$.
In particular constants $\kappa_0$, $\zeta_0$, $\CF$, and $\CT$ are not needed. 

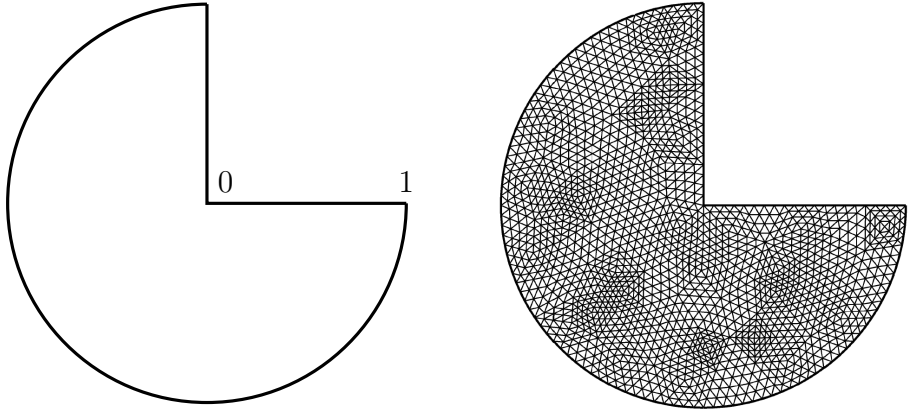
\begin{figure}
\begin{center}
\begin{tikzpicture}[scale=2.65]
\draw[very thick] (0,1)--(0,0)--(1,0);
\draw[very thick] (0,1) arc (90:360:1); 
\node [above right] at (0,0) {$0$};
\node [above] at (1,0) {$1$};
\end{tikzpicture}
\qquad
\raisebox{-1mm}{\includegraphics[width=0.40\textwidth]{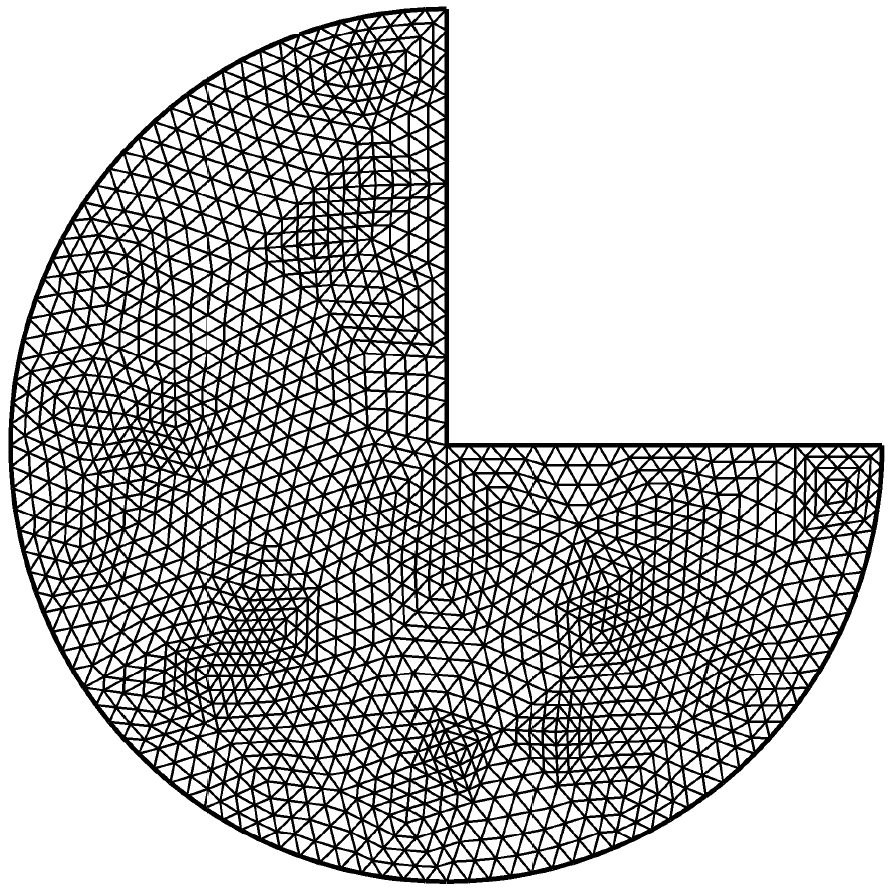}}
\end{center}
\caption{The domain $\Omega$ (left) and the uniform mesh (right) used in Example~1.}
\label{fi:mesh}
\end{figure}

%
%
%

Figure~\ref{fi:Ieff} (left) presents the index of effectivity 
\begin{equation}
  \label{eq:Ieff}
  \Ieff = \frac{\eta(\btau)}{\trinorm{ u - u_h }}
\end{equation}
for the chosen values of $\kappa$. All values of $\Ieff$ are above 1 confirming that $\eta(\btau)$ is the guaranteed upper bound on the error. On the other hand they are not far from 1 in the whole range of values of $\kappa$ showing the robust efficiency. All these indices of effectivity are below 1.12, which illustrates high accuracy of computed error estimators.
For comparison, we also present indices of effectivity for the error estimator proposed in our previous work \cite{AinVej:2014}, see the dashed lines in Figure~\ref{fi:Ieff}. Its accuracy legs behind the current approach.

To illustrate the robustness with respect to the mesh size, we also solve this problem on a sequence of uniformly refined meshes for a fixed value $\kappa = 100$ and plot the resulting indices of effectivity in Figure~\ref{fi:Ieff} (right). In this case we observe robust efficiency and high accuracy as well. 

\begin{figure}
\includegraphics[width=0.48\textwidth]{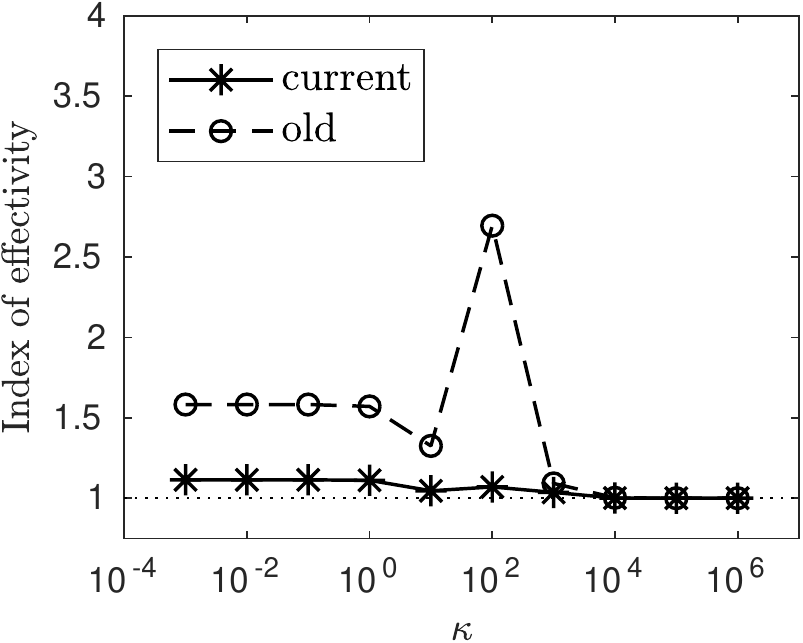}%
\quad
\includegraphics[width=0.48\textwidth]{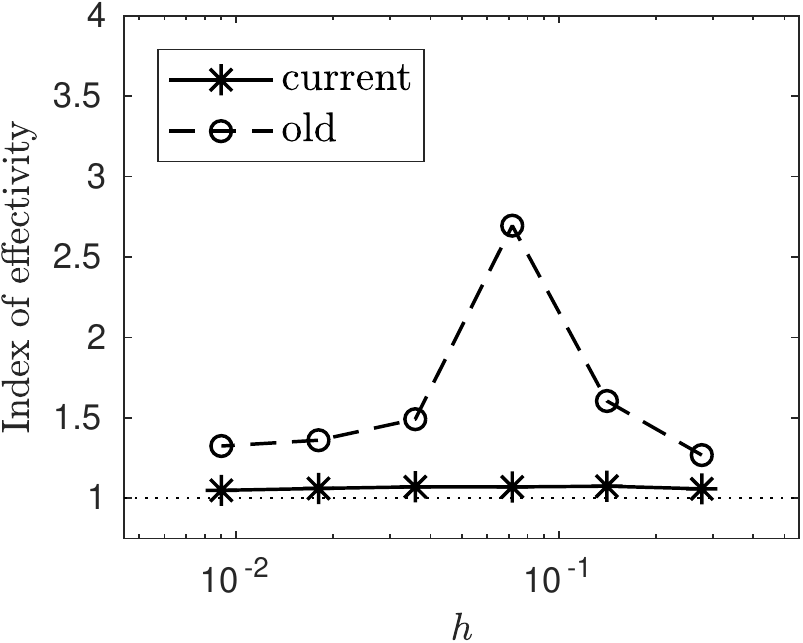}%
\caption{Indices of effectivity with respect to $\kappa$ (left) and $h$ (right) on uniformly refined meshes for Example~1.
Solid lines present the current estimator $\eta(\btau)$ while dashed lines the estimator from \cite{AinVej:2014}.}
\label{fi:Ieff}
\end{figure}

Error indicators $\eta_K(\btau)$ given in \eqref{eq:etaK} can be utilized for adaptive mesh refinement and error estimator $\eta(\btau)$ for a guaranteed stopping criterion.  We use the standard adaptive algorithm: \texttt{SOLVE} -- \texttt{ESTIMATE} -- \texttt{STOP} -- \texttt{MARK} -- \texttt{REFINE}.
Given an initial mesh, the \texttt{SOLVE} step computes the finite element solution by \eqref{eq:FEM},
the \texttt{ESTIMATE} step evaluates the flux reconstruction $\btau$ defined by \eqref{eq:tau} and error indicators $\eta_K(\btau)$ introduced in \eqref{eq:etaK}.
In the \texttt{STOP} step, the error estimator $\eta(\btau)$ given by \eqref{eq:eta} is computed and the algorithm is stopped
if $\eta(\btau)$ (and consequently the error $\trinorm{u-u_h}$) is below the required tolerance.
In the \texttt{MARK} step, the D\"orfler strategy \cite{Dorfler:1996} is used to mark elements, where $\eta_K(\btau)$ indicate large error. 
Finally, the longest edge bisection algorithm \cite{Mitchell:1989,Verfurth:1996} 
%
is applied in the \texttt{REFINE} step to refine the marked elements and create a new mesh.

Several examples of adaptively refined meshes are provided in Figure~\ref{fi:meshadap}.
The optimal speed of convergence of both the error $\trinorm{u-u_h}$ and error estimator $\eta(\btau)$ during the adaptive algorithm is presented in Figure~\ref{fi:adap} (left).
Figure~\ref{fi:adap} (right) shows corresponding indices of effectivity. They are all above and quite close to 1, confirming the robust efficiency of the error estimator even on highly graded meshes.

\begin{figure}
\includegraphics[width=0.31\textwidth]{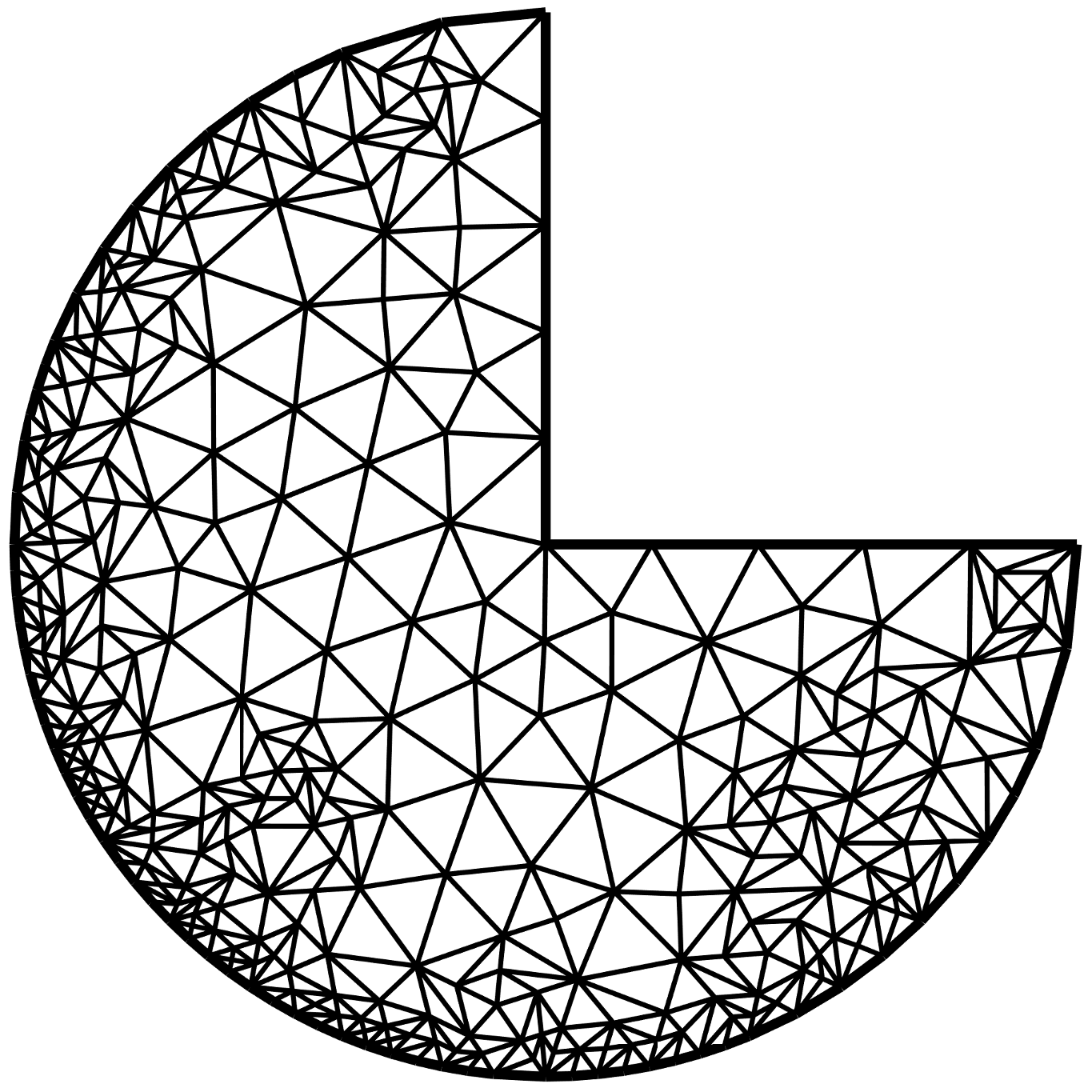}\quad
\includegraphics[width=0.31\textwidth]{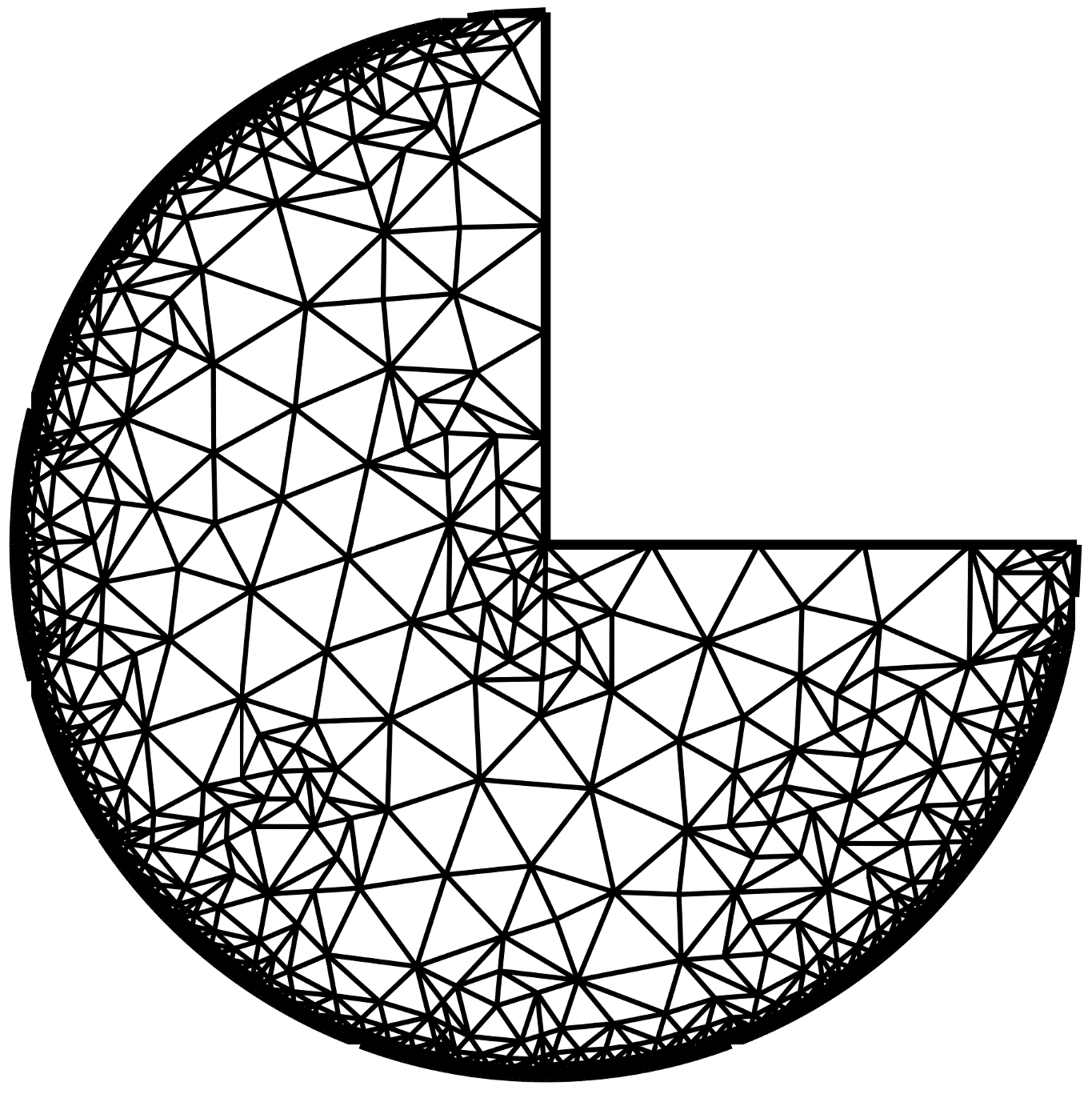}\quad
\includegraphics[width=0.31\textwidth]{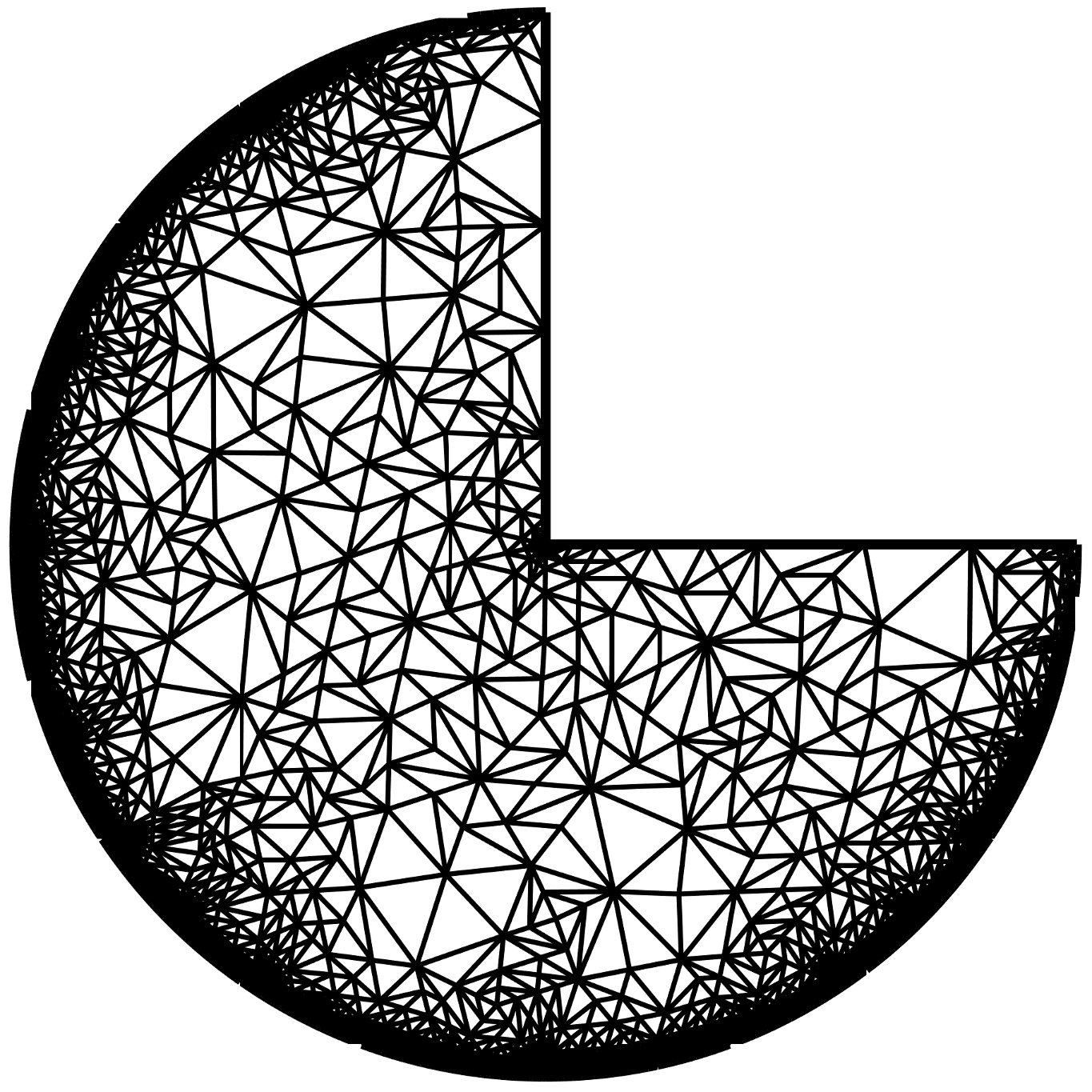}
\caption{Adaptively refined meshes after 10 (left), 30 (middle), and 40 (right) refinement steps in Example~1.}
\label{fi:meshadap}
\end{figure}

\begin{figure}
\includegraphics[width=0.48\textwidth]{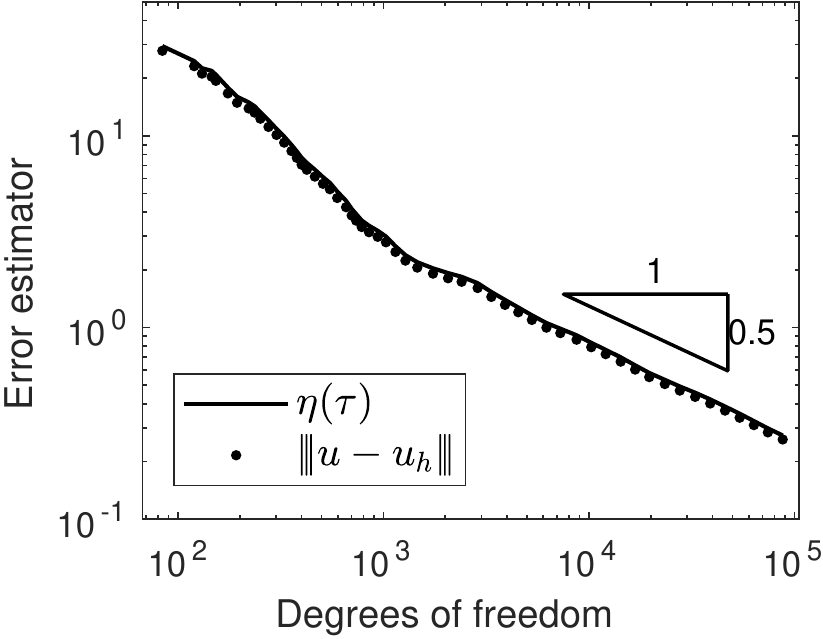}%
\quad
\includegraphics[width=0.48\textwidth]{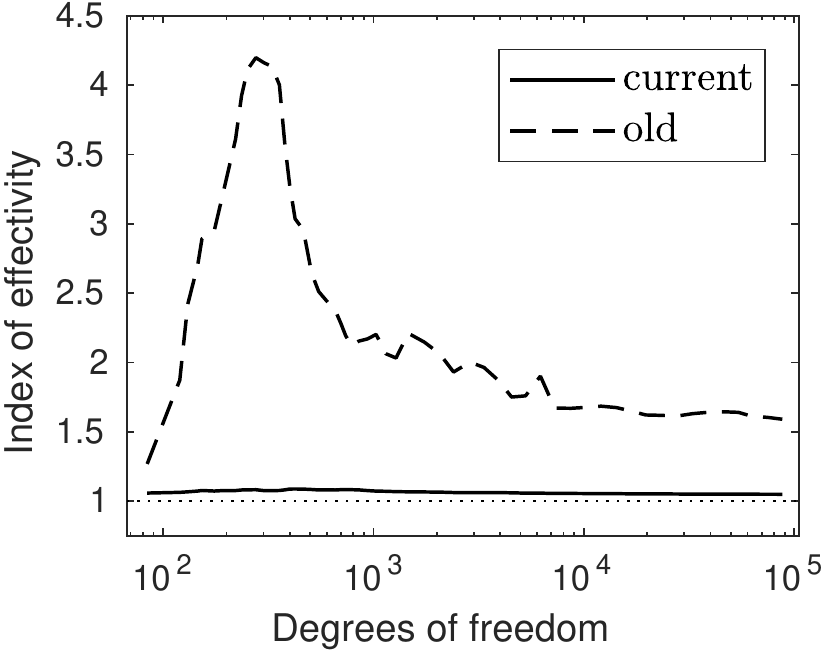}%
\caption{Convergence of the energy norm of the error and the error estimator during the adaptive algorithm (left) for Example~1 and $\kappa=100$. Corresponding indices of effectivity (right). The dashed line presents the estimator from \cite{AinVej:2014}.}
\label{fi:adap}
\end{figure}

\paragraph*{Example 2}
This example illustrates the behaviour of the proposed error estimator for a discontinuous right-hand side $f$, piecewise constant reaction coefficient $\kappa$, and homogeneous Neumann boundary conditions.
We consider problem~\eqref{eq:modpro} in a square $\Omega = (-1,1)^2$ with $\GammaN = \partial\Omega$ and $\gN = 0$ on $\GammaN$. Right-hand side $f$ equals to $\kappa^2$ in the disc $B_{1/2} = \{\varrho \leq 1/2\}$, where $\varrho$ is the distance from the origin, and it vanishes elsewhere. The exact solution of this problem is not known, but for large $\kappa$ it is supposed to be close to the characteristic function of the disc $B_{1/2}$ with a steep interior layer close to the boundary of $B_{1/2}$.
Since the exact solution is not known, we approximate the true error by $u_h^\mathrm{ref} - u_h$, where the reference solution $u_h^\mathrm{ref}$ is computed by finite elements of order 5 on the same mesh as $u_h$.

We choose $\kappa=100$ and use the modified flux reconstruction \eqref{eq:tauB} and the modified error estimator \eqref{eq:etaB}. Since $\Omega$ is a square, we can compute the Friedrichs--Poincar\'e and trace constants analytically. We use $\CF^2 = 2/\pi^2$ and $\CT^2 = \sqrt{2}\coth( \sqrt{2}/2)$.
Parameters $\kappa_0$ and $\zeta_0$ are chosen as square roots of the machine epsilon: $\kappa_0 = \zeta_0 \approx 10^{-8}$.

We solve this problem by the adaptive algorithm described above starting with a mesh with two triangles.
This setting does not satisfy assumptions listed in Subsection~\ref{se:kappa}, because discontinuities in $\kappa$ are not compatible with the mesh. Therefore, for the purpose of computation, we use the value of $\kappa$ in the centroid of each element as the constant value in the element. In this way we construct certain approximate solution and the corresponding error estimator, which is guaranteed by Theorem~\ref{th:mainB} to be above the true error. The obtained indices of effectivity show robust and efficient performance of the estimator even in this case.

Figure~\ref{fi:adap2} (left) shows the energy norm of the approximate error $\trinorm{u_h^\mathrm{ref} - u_h}$, the computed error bound $\teta(\tbtau)$, and the oscillation term 
$
  \osc^2(f,\gN) = \sum_{K\in\cT_h} \osc_K^2(f,\gN)
$
during the adaptive process. Figure~\ref{fi:adap2} (right) presents the corresponding indices of effectivity 
$\Ieff = \teta(\tbtau)/\trinorm{u_h^\mathrm{ref} - u_h}$. We may observe that the error bound is really above the error and that the error estimator estimates it robustly on all meshes. The oscillation term is of comparable size as the error at the beginning of the adaptive process, which leads to higher values of the index of effectivity. However, starting from meshes with around $10^3$ degrees of freedom the interior layer is well resolved, the oscillation term decreases faster than the error, and the index of effectivity decreases towards one. 
For illustration we present three adaptively refined meshes in Figure~\ref{fi:meshadap2}.

\begin{figure}
\includegraphics[width=0.48\textwidth]{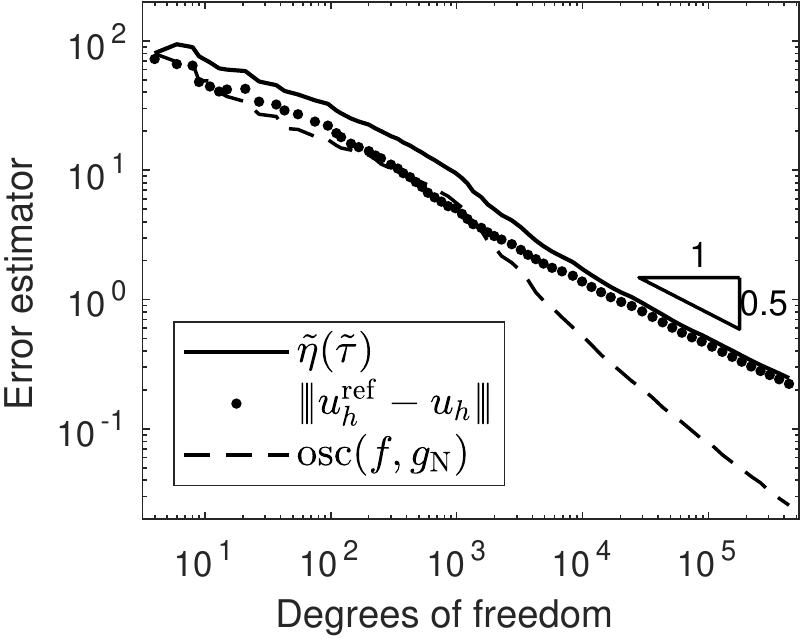}%
\quad
\includegraphics[width=0.48\textwidth]{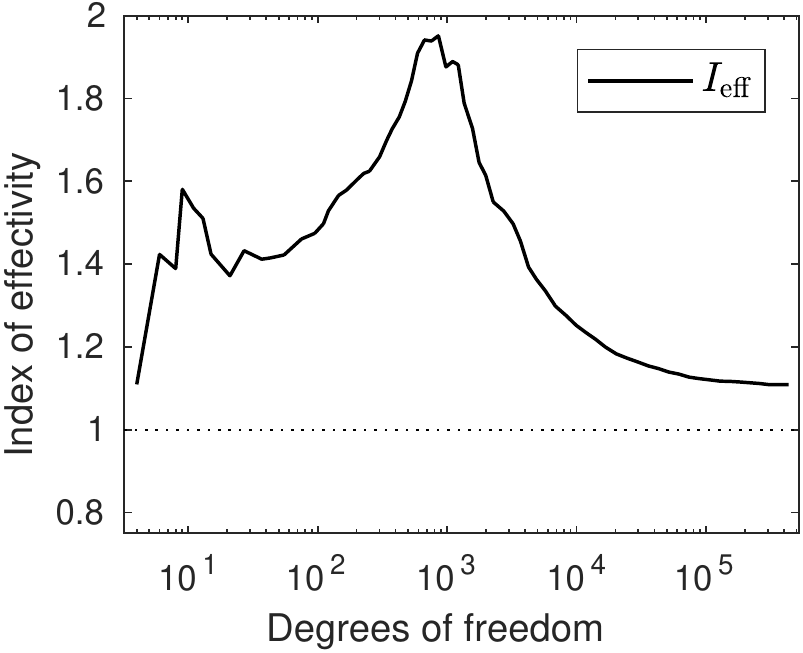}%
\caption{Convergence of the energy norm of the error and the error estimator during the adaptive algorithm (left) for Example~2 and $\kappa=100$. Corresponding indices of effectivity (right).}
\label{fi:adap2}
\end{figure}

\begin{figure}
\includegraphics[width=0.31\textwidth]{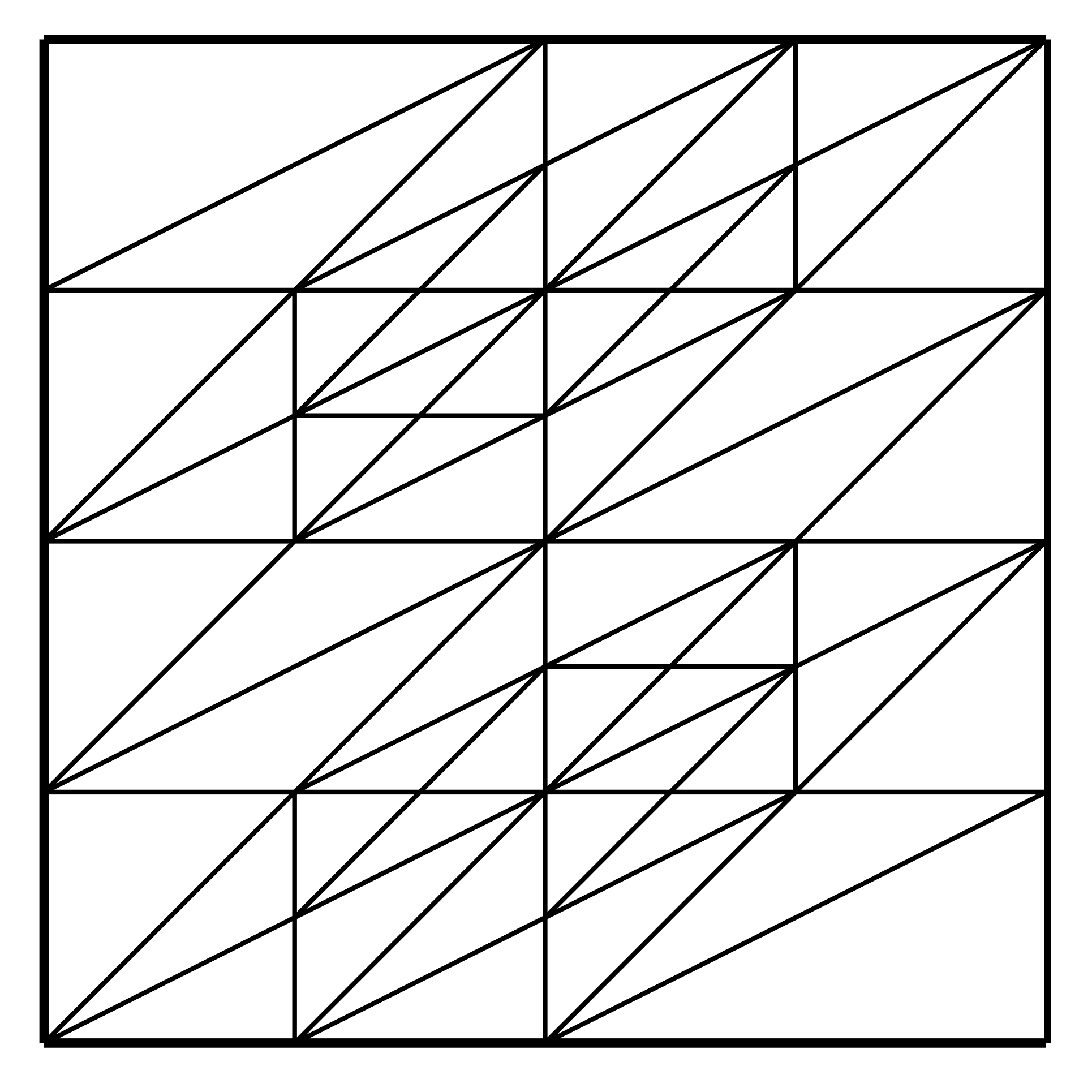}\quad
\includegraphics[width=0.31\textwidth]{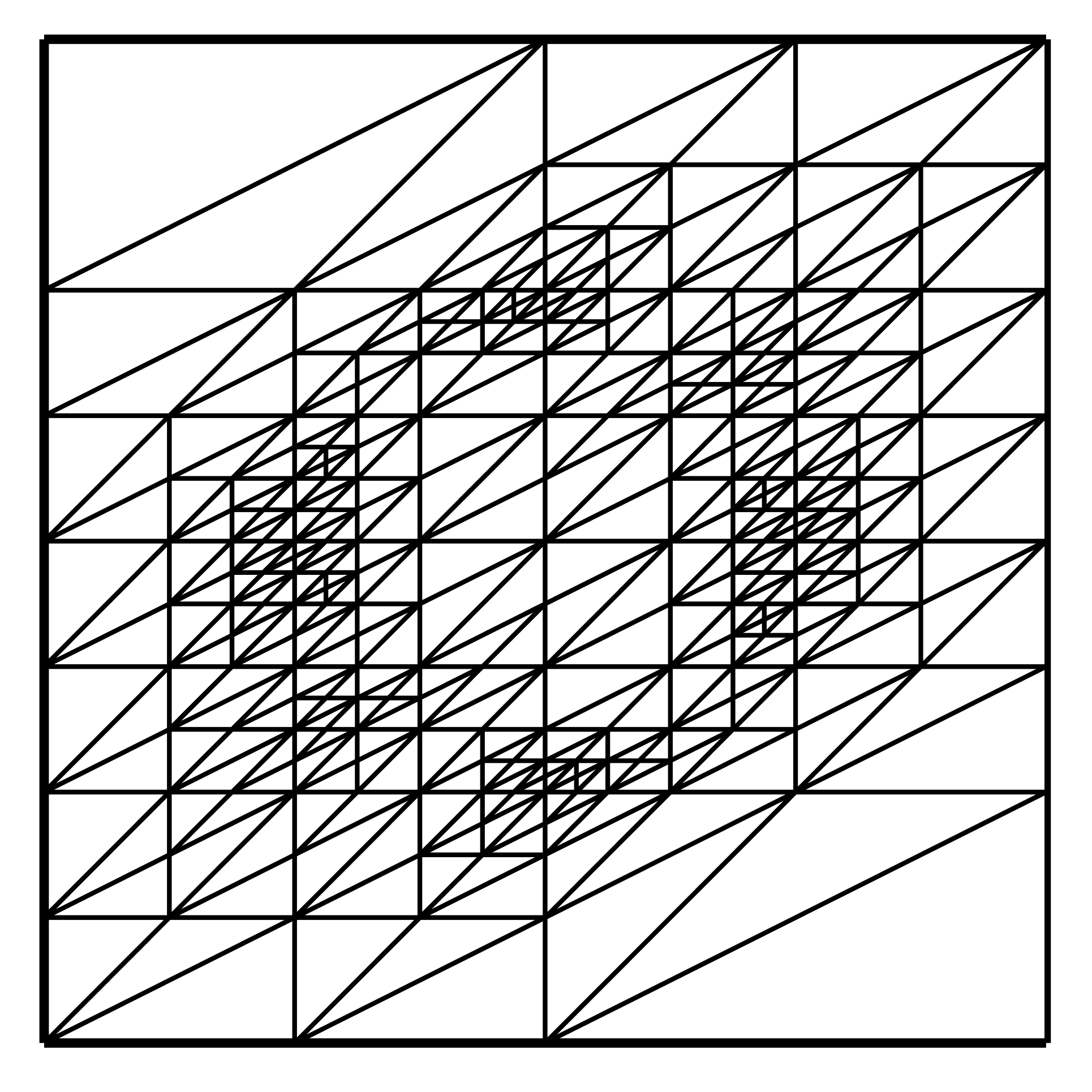}\quad
\includegraphics[width=0.31\textwidth]{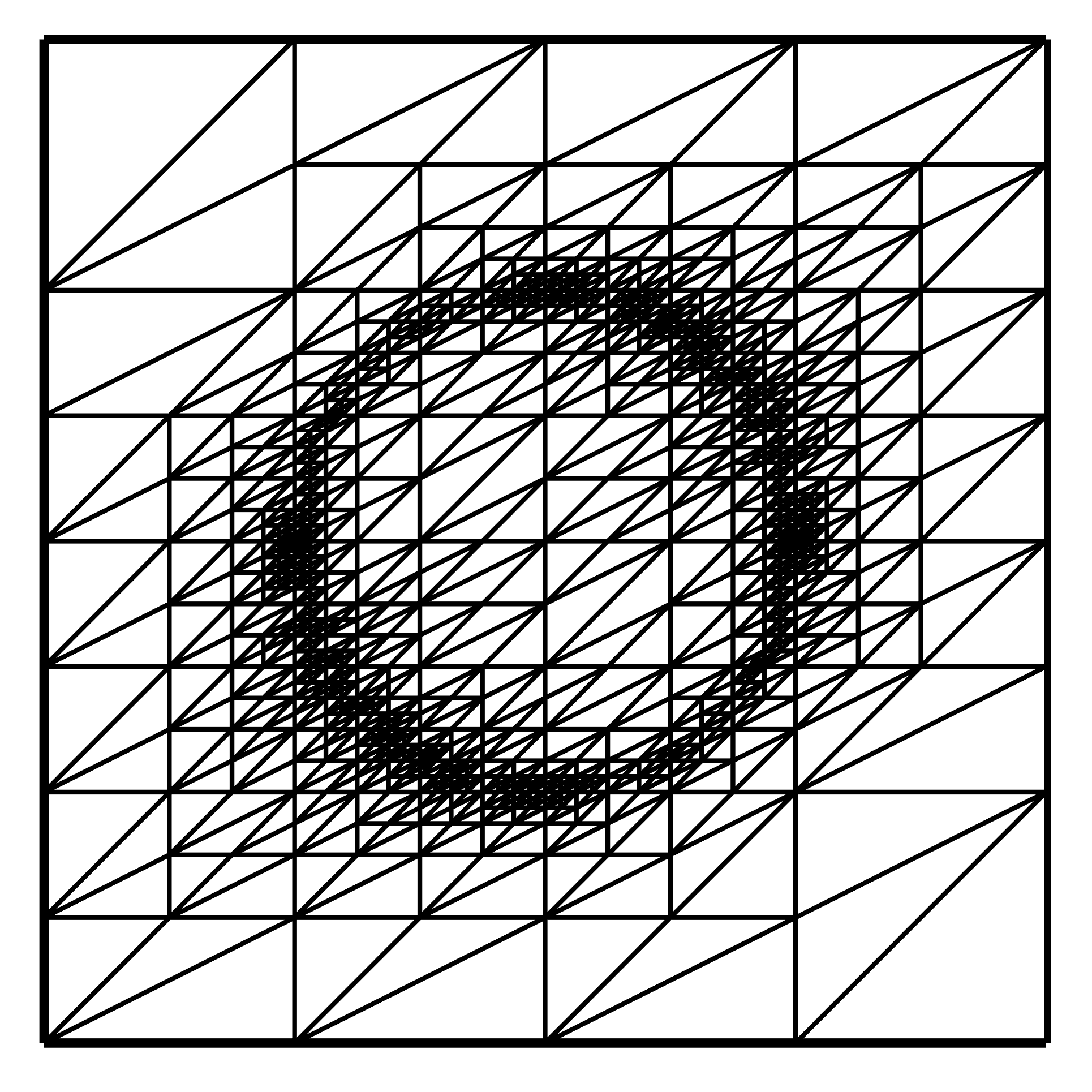}\quad
\caption{Adaptively refined meshes after 10 (left), 20 (middle), and 30 (right) refinement steps in Example~2.}
\label{fi:meshadap2}
\end{figure}

\section{Conclusions}
\label{se:conclusions}

In this paper we present an a posteriori error estimator that is fully computable and provides a locally efficient upper bound on the energy norm of the error. This error estimator can be computed by a fast and easily parallelizable algorithm by solving small and independent problems on patches of elements. We proved its robustness both with respect to the mesh size and the reaction coefficient $\kappa$. We demonstrated by numerical examples that the corresponding local error indicators can be successfully used in the standard adaptive algorithm to guide the mesh adaptation and that the error estimator
provides sharp results on rough, fine, and adaptively refined meshes as well as in the singularly perturbed case when $\kappa$ is large.

Further research questions about this error estimator may include its robustness for higher order finite element approximations \cite{SmeVoh:2018} and its possible modifications to guarantee robustness on anisotropically refined meshes.

The proposed flux reconstruction can be used not only for the presented reaction-diffusion problems, but also for related eigenvalue problems. It was recently shown \cite{Vejchodsky:2018} that any flux reconstruction for boundary value problems can be directly used in the Lehmann--Goerisch method for guaranteed bounds on eigenvalues.




\begin{thebibliography}{10}

\bibitem{AinBab:1999}
Ainsworth, M. and Babu{\v{s}}ka, I.: Reliable and robust a posteriori error
  estimating for singularly perturbed reaction-diffusion problems.
\newblock SIAM J. Numer. Anal. \textbf{36} (1999), 331--353 (electronic).

\bibitem{AinOde:2000}
Ainsworth, M. and Oden, J.T.: \emph{A posteriori error estimation in finite
  element analysis}.
\newblock Pure and Applied Mathematics (New York), Wiley-Interscience [John
  Wiley \& Sons], New York, 2000.

\bibitem{robustaee:2010}
Ainsworth, M. and Vejchodsk{\'y}, T.: Fully computable robust a posteriori
  error bounds for singularly perturbed reaction--diffusion problems.
\newblock Numer. Math. \textbf{119} (2011), 219--243.

\bibitem{AinVej:2014}
Ainsworth, M. and Vejchodsk{\'y}, T.: Robust error bounds for finite element
  approximation of reaction-diffusion problems with non-constant reaction
  coefficient in arbitrary space dimension.
\newblock Comput. Methods Appl. Mech. Engrg. \textbf{281} (2014), 184--199.

\bibitem{AubBur:1971}
Aubin, J.P. and Burchard, H.G.: Some aspects of the method of the hypercircle
  applied to elliptic variational problems.
\newblock In: \emph{Numerical {S}olution of {P}artial {D}ifferential
  {E}quations, {II} ({SYNSPADE} 1970) ({P}roc. {S}ympos., {U}niv. of
  {M}aryland, {C}ollege {P}ark, {M}d., 1970)}, pp. 1--67. Academic Press, New
  York, 1971.

\bibitem{BraSch:2008}
Braess, D. and Sch{\"o}berl, J.: Equilibrated residual error estimator for edge
  elements.
\newblock Math. Comp. \textbf{77} (2008), 651--672.

\bibitem{BreFor:1991}
Brezzi, F. and Fortin, M.: \emph{Mixed and hybrid finite element methods}.
\newblock Springer-Verlag, New York, 1991.

\bibitem{CaiZha:2010}
Cai, Z. and Zhang, S.: Flux recovery and a posteriori error estimators:
  conforming elements for scalar elliptic equations.
\newblock SIAM J. Numer. Anal. \textbf{48} (2010), 578--602.

\bibitem{CheFucPriVoh:2009}
Cheddadi, I., Fu{\v{c}}{\'{\i}}k, R., Prieto, M.I., and Vohral{\'{\i}}k, M.:
  Guaranteed and robust a posteriori error estimates for singularly perturbed
  reaction--diffusion problems.
\newblock M2AN Math. Model. Numer. Anal. \textbf{43} (2009), 867--888.

\bibitem{DolSebVoh:2015}
Dolej\v{s}{\'\i}, V., \v{S}ebestov\'a, I., and Vohral{\'\i}k, M.: Algebraic and
  discretization error estimation by equilibrated fluxes for discontinuous
  {G}alerkin methods on nonmatching grids.
\newblock J. Sci. Comput. \textbf{64} (2015), 1--34.

\bibitem{Dorfler:1996}
D{\"o}rfler, W.: A convergent adaptive algorithm for {P}oisson's equation.
\newblock SIAM J. Numer. Anal. \textbf{33} (1996), 1106--1124.

\bibitem{ErnVoh:2010}
Ern, A. and Vohral{\'\i}k, M.: A posteriori error estimation based on potential
  and flux reconstruction for the heat equation.
\newblock SIAM J. Numer. Anal. \textbf{48} (2010), 198--223.

\bibitem{Grosman:2006}
Grosman, S.: An equilibrated residual method with a computable error
  approximation for a singularly perturbed reaction-diffusion problem on
  anisotropic finite element meshes.
\newblock M2AN Math. Model. Numer. Anal. \textbf{40} (2006), 239--267.

\bibitem{HanSteVoh:2012}
Hannukainen, A., Stenberg, R., and Vohral{\'\i}k, M.: A unified framework for a
  posteriori error estimation for the {S}tokes problem.
\newblock Numer. Math. \textbf{122} (2012), 725--769.

\bibitem{HasHla:1976}
Haslinger, J. and Hlav{\'a}{\v{c}}ek, I.: Convergence of a finite element
  method based on the dual variational formulation.
\newblock Apl. Mat. \textbf{21} (1976), 43--65.

\bibitem{JirStrVoh:2010}
Jir{\'a}nek, P., Strako{\v{s}}, Z., and Vohral{\'{\i}}k, M.: A posteriori error
  estimates including algebraic error and stopping criteria for iterative
  solvers.
\newblock SIAM J. Sci. Comput. \textbf{32} (2010), 1567--1590.

\bibitem{Kelly:1984}
Kelly, D.W.: The self-equilibration of residuals and complementary a posteriori
  error estimates in the finite element method.
\newblock Internat. J. Numer. Methods Engrg. \textbf{20} (1984), 1491--1506.

\bibitem{Kopteva:2017}
Kopteva, N.: Energy-norm a posteriori error estimates for singularly perturbed
  reaction-diffusion problems on anisotropic meshes.
\newblock Numer. Math. \textbf{137} (2017), 607--642.

\bibitem{Kopteva:2018}
Kopteva, N.: Fully computable a posteriori error estimator using anisotropic
  flux equilibration on anisotropic meshes.
\newblock Preprint arXiv:1704.04404  (2017), 32 p.

\bibitem{LadLeg:1983}
Ladev{\`e}ze, P. and Leguillon, D.: Error estimate procedure in the finite
  element method and applications.
\newblock SIAM J. Numer. Anal. \textbf{20} (1983), 485--509.

\bibitem{LucWoh:2004}
Luce, R. and Wohlmuth, B.I.: A local a posteriori error estimator based on
  equilibrated fluxes.
\newblock SIAM J. Numer. Anal. \textbf{42} (2004), 1394--1414.

\bibitem{Mitchell:1989}
Mitchell, W.F.: A comparison of adaptive refinement techniques for elliptic
  problems.
\newblock ACM Trans. Math. Software \textbf{15} (1989), 326--347 (1990).

\bibitem{PapStrVoh:2018}
Pape\v{z}, J., Strako\v{s}, Z., and Vohral{\'\i}k, M.: Estimating and
  localizing the algebraic and total numerical errors using flux
  reconstructions.
\newblock Numer. Math. \textbf{138} (2018), 681--721.

\bibitem{ParDie:2017}
Par\'es, N. and D{\'\i}ez, P.: A new equilibrated residual method improving
  accuracy and efficiency of flux-free error estimates.
\newblock Comput. Methods Appl. Mech. Engrg. \textbf{313} (2017), 785--816.

\bibitem{ParSanDie2009}
Par{\'e}s, N., Santos, H., and D{\'{\i}}ez, P.: Guaranteed energy error bounds
  for the {P}oisson equation using a flux-free approach: solving the local
  problems in subdomains.
\newblock Internat. J. Numer. Methods Engrg. \textbf{79} (2009), 1203--1244.

\bibitem{PraSyn:1947}
Prager, W. and Synge, J.L.: Approximations in elasticity based on the concept
  of function space.
\newblock Quart. Appl. Math. \textbf{5} (1947), 241--269.

\bibitem{Rep:2008}
Repin, S.: \emph{A posteriori estimates for partial differential equations},
  \emph{Radon Series on Computational and Applied Mathematics}, vol.~4.
\newblock de Gruyter, Berlin, 2008.

\bibitem{SebVej:2014}
{\v{S}}ebestov{\'a}, I. and Vejchodsk{\'y}, T.: Two-sided bounds for
  eigenvalues of differential operators with applications to {F}riedrichs,
  {P}oincar\'e, trace, and similar constants.
\newblock SIAM J. Numer. Anal. \textbf{52} (2014), 308--329.

\bibitem{SmeVoh:2018}
Smears, I. and Vohral{\'\i}k, M.: Simple and robust equilibrated flux a posteriori estimates for singularly
perturbed reaction-diffusion problems. 
\newblock Preprint hal-01956180, 2018.

\bibitem{Synge:1957}
Synge, J.L.: \emph{The hypercircle in mathematical physics: a method for the
  approximate solution of boundary value problems}.
\newblock Cambridge University Press, New York, 1957.

\bibitem{Vejchodsky:2018}
Vejchodsk\'y, T.: Flux reconstructions in the {L}ehmann-{G}oerisch method for
  lower bounds on eigenvalues.
\newblock J. Comput. Appl. Math. \textbf{340} (2018), 676--690.

\bibitem{Verfurth:1998}
Verf{\"u}rth, R.: A posteriori error estimators for convection-diffusion
  equations.
\newblock Numer. Math. \textbf{80} (1998), 641--663.

\bibitem{Verfurth:1998a}
Verf\"urth, R.: Robust a posteriori error estimators for a singularly perturbed
  reaction-diffusion equation.
\newblock Numer. Math. \textbf{78} (1998), 479--493.

\bibitem{Verfurth:1996}
Verf{\"u}rth, R.: \emph{{A review of a posteriori error estimation and adaptive
  mesh-refinement techniques.}}
\newblock Wiley-Teubner, Chichester/Stuttgart, 1996.

\bibitem{Veubeke:1965}
de~Veubeke, B.F.: Displacement and equilibrium models in the finite element
  method.
\newblock In: O.~Zienkiewicz and G.~Hollister (Eds.), \emph{Stress Analysis},
  pp. 145--197. Wiley, London, 1965.

\end{thebibliography}



\end{document}